\documentclass[a4paper, 11pt]{amsart}

\usepackage{comment}
\usepackage{graphicx}
\usepackage{tikz-cd}
\input xy
\xyoption{all}

\title[Stability conditions via mass of spherical objects]
{Stability conditions on K3 surfaces via mass of spherical objects}
\author{Kohei Kikuta, Naoki Koseki, and Genki Ouchi}
\date{}

\address{Department of Mathematics, Graduate School of Science, Osaka University, Toyonaka Osaka, 560-0043, Japan.
\newline
\hspace*{4mm}
School of Mathematics, 
The University of Edinburgh, 
James Clerk Maxwell Building, 
Peter Guthrie Tait Road, Edinburgh, Scotland, EH9 3FD, UK.}
\email{kikuta@math.sci.osaka-u.ac.jp}

\address{The University of Liverpool, Mathematical Sciences Building, Liverpool, L69 7ZL, UK.}
\email{koseki@liverpool.ac.uk}

\address{Department of Mathematics, Faculty of Science, Hokkaido University, Kita 10, Nishi 8, Kita-Ku, Sapporo, Hokkaido, 060-0810, Japan.}
\email{genki.ouchi@math.sci.hokudai.ac.jp}

\address{}
\email{}

\makeatletter
 
  \@addtoreset{equation}{section}
 \makeatother

\usepackage{amsmath, amssymb, amsthm, amscd, comment, mathtools, color}
\usepackage{todonotes}
\usepackage[frame,cmtip,curve,arrow,matrix,line,graph]{xy}
\usepackage[mathscr]{euscript}
\usepackage{mathrsfs}
\usepackage{tikz}
\usetikzlibrary{intersections, calc}

\theoremstyle{plain}
\newtheorem{thm}{Theorem}[section]
\newtheorem{prop}[thm]{Proposition}
\newtheorem{def-prop}[thm]{Definition-Proposition}
\newtheorem{lem}[thm]{Lemma}
\newtheorem{cor}[thm]{Corollary}
\newtheorem*{thm*}{Theorem}

\theoremstyle{definition}
\newtheorem{dfn}[thm]{Definition}

\newtheorem*{NaC}{Notation and Convention}
\newtheorem*{ACK}{Acknowledgement}

\theoremstyle{remark}
\newtheorem{rmk}[thm]{Remark}

\newtheorem{ques}[thm]{Question}
\newtheorem{ex}[thm]{Example}

\DeclareMathOperator{\ch}{ch}
\DeclareMathOperator{\td}{td}

\DeclareMathOperator{\id}{id}

\newcommand{\bP}{\mathbb{P}}
\newcommand{\bC}{\mathbb{C}}

\newcommand{\bR}{\mathbb{R}}
\newcommand{\bQ}{\mathbb{Q}}
\newcommand{\bZ}{\mathbb{Z}}
\newcommand{\bH}{\mathbb{H}}
\newcommand{\mcA}{\mathcal{A}}
\newcommand{\mcB}{\mathcal{B}}

\newcommand{\mcD}{\mathcal{D}}

\newcommand{\mcF}{\mathcal{F}}

\newcommand{\mcO}{\mathcal{O}}
\newcommand{\mcP}{\mathcal{P}}
\newcommand{\mcQ}{\mathcal{Q}}

\newcommand{\mcS}{\mathcal{S}}
\newcommand{\mcT}{\mathcal{T}}

\DeclareMathOperator{\Hom}{Hom}

\DeclareMathOperator{\Coh}{Coh}

\DeclareMathOperator{\Ker}{Ker}
\DeclareMathOperator{\Coker}{Coker}
 
\DeclareMathOperator{\NS}{NS}

\DeclareMathOperator{\ext}{ext}
\DeclareMathOperator{\Ext}{Ext}

\DeclareMathOperator{\Stab}{Stab}
\DeclareMathOperator{\Slice}{Slice}

\DeclareMathOperator{\cl}{cl}

\DeclareMathOperator{\GL}{GL}

\DeclareMathOperator{\Cone}{Cone}

\newcommand{\Halg}{\widetilde{H}^{1,1}}

\DeclareMathOperator{\Aut}{Aut}

\DeclareMathOperator{\Obj}{Obj}

\newcommand{\iso}{\xrightarrow{\sim}}

\begin{document}

\begin{abstract}
We prove that a stability condition on a K3 surface is determined by the masses of spherical objects up to a natural $\bC$-action. 
This is motivated by the result of Huybrechts and the recent proposal of Bapat--Deopurkar--Licata on the construction of a compactification of a stability manifold. 
We also construct lax stability conditions in the sense of Broomhead--Pauksztello--Ploog--Woolf associated to spherical bundles. 
\end{abstract}

\maketitle

\setcounter{tocdepth}{1}
\tableofcontents

\section{Introduction}
\subsection{Motivation and Results}
Initiated by the works of Gaiotto--Moore--Neitzke \cite{GMN} and Bridgeland--Smith \cite{BS}, the analogy between the theory of Bridgeland stability conditions and the classical Teichm{\"u}ller theory has been explored in various interesting settings. Motivated by those works, Bapat--Deopurkar--Licata \cite{bdl20} proposed a way to
solve a long-standing open problem of constructing a natural compactification of the space of Bridgeland stability conditions. Let us review their proposal here. Let $\mcD$ be a triangulated category. The space $\Stab(\mcD)$ of Bridgeland stability conditions on $\mcD$ has a natural $\bC$-action. For a fixed set of isomorphism classes of objects $\mcS \subset \Obj(\mcD)$, consider a map to the real projective space (of possibly infinite dimension): 
\begin{equation} \label{eq:mass-intro}
m \colon \Stab(\mcD)/\bC \to \bP^\mcS, \quad [\sigma] \mapsto [(m_\sigma(E))_{E \in \mcS}], 
\end{equation}
where $m_\sigma(E)$ is the so-called mass of an object $E$ with respect to a stability conditions $\sigma$. Then one can ask: 
\begin{ques} \label{ques:homeo}
When is the map $m$ a homeomorphism onto the image?
\end{ques}
In \cite{bdl20}, the authors proved that the closure of the image of $m$ is compact under the assumption that $\mcD$ has a split generator. Hence in that case, the affirmative answer to Question \ref{ques:homeo} would yield a compactification of the space $\Stab(\mcD)/\bC$. 
In our previous work \cite{kko24}, we have given the complete answer to Question \ref{ques:homeo} when $\mcD$ is the derived category of coherent sheaves on a smooth projective curve. The first aim of the present article is to give a partial answer to this question for a K3 surface. 

Let $X$ be a K3 surface, and $\mcD=D^b(X)$. In view of the analogy with the classical Teichm{\"u}ller theory, it is natural to take $\mcS$ to be the set of isomorphism classes of spherical objects on $D^b(X)$. Our first main result is: 
\begin{thm}[Theorem \ref{thm:inj}] \label{thm:inj-intro}
Let $X$ be a K3 surface, $\mcS$ the set of isomorphism classes of spherical objects on $D^b(X)$. Then the mass map (\ref{eq:mass-intro}), restricted to the distinguished component $\Stab^*(X)/\bC$, is injective. 
\end{thm}
This is already a surprising result: we can recover stability conditions (up to the $\bC$-action) only from the information of spherical objects. A similar result was previously proved by Huybrechts \cite{huy12}. Indeed, Huybrechts' result will play a key role in the proof of Theorem \ref{thm:inj-intro}. 
At this moment, we do not know if the mass map in this case is a homeomorphism onto the image or not. 

\begin{rmk}
In \cite{bdl20}, they introduced a variant $m'$ of the mass map which also encodes information of phases of objects. They proved in a very general setting that the product $m \times m' \colon \Stab(\mcD)/\bC \to \bP^\mcS \times \bP^\mcS$ is always injective. 
In the case of a K3 surface, we obtained a stronger result without using $m'$. 
\end{rmk}

Let $M$ denote the closure of $m(\Stab^*(X)/\bC) \subset \bP^\mcS$. It would be also interesting to investigate points in the boundary of $M$. 
For a stability condition $\sigma$, the mass $m_\sigma(E)$ is always positive. Natural points in the boundary of $M$ would appear if we look at limits where the mass of a fixed spherical object $A$ becomes zero. 
There are two different types of such limits: 
\begin{enumerate}
\item A hom-functional $h_A =[(\overline{\hom}(A, E))_{E \in \mcS}] \in \bP^\mcS$, where $\overline{\hom}(A, E)$ is defined as follows: 
\[
\overline{\hom}(A, E) \coloneqq \begin{cases}
\sum_{i \in \bZ} \dim\Ext^i(A, E) & (E\neq A[i]\text{ for any }i\in\bZ), \\
0 & (otherwise). 
\end{cases}
\]
\item A lax stability condition, which is a weaker notion of stability conditions, allowing the mass of the spherical object $A$ to be zero. 
\end{enumerate}
The hom-functional $h_A$ was introduced in \cite{bdl20}, and they showed that 
$h_A$ lies on the boundary $\partial M$. 
The notion of lax stability conditions was introduced in \cite{bppw} in order to study a partial compactification of the space of stability conditions. 
Our second main result is the following: 

\begin{thm}[Theorems \ref{thm:laxstab} and \ref{prop:laxvshom}] 
There is a spherical bundle $A$ on $X$ and a lax stability condition $\sigma_0$ on $D^b(X)$ satisfying the following properties:
\begin{enumerate}
\item $A$ is massless, equivalently, we have $m_{\sigma_0}(A)=0$. 
\item We have $m(\sigma_0) \in \partial M$.
\item We have $m(\sigma_0) \neq h_A$ in $\bP^\mcS$.
\end{enumerate}
\end{thm}
In fact, we will construct lax stability conditions associated to any spherical classes satisfying mild numerical assumptions. When $\rho(X)=1$, Mukai vectors of any spherical bundles satisfy those assumptions. See Subsection \ref{sec:constructlax} for more details.


\subsection{Relation with other works}
There have been various attempts to construct a natural (partial) compactification of the space of Bridgeland stability conditions \cite{bdl20, bol23, bppw, kko24, halpernleistner25}. Most of the works in the literature consider examples arising from quivers. 
In our previous work \cite{kko24}, the first algebro-geometric case of curves has been worked out. The current article provides the first progress in this direction for higher dimensional varieties. 

While writing the first draft of this article, Laura Pertusi informed us that they obtained a similar result as in Theorem \ref{thm:inj-intro} for a K3 surface of Picard rank one. 

In \cite{clsy24_2, clsy24, TYChou}, the authors constructed some examples of weak stability conditions on surfaces. 

\subsection{Future directions}
In the case of a K3 surface, we currently do not know if the mass map is a homeomorphism onto the image. This is an important open problem towards the construction of a compactification of $\Stab^*(X)/\bC$. 
It would be also interesting to apply our results to the study of the autoequivalence group of a K3 surface by using its action on the closure $M$. 
In the case of an elliptic curve, we gave a Nielsen--Thurston type classification of autoequivalences in terms of the behavior of fixed points of this action in \cite{kko24}. 
See \cite{fl23} for the progress in this direction for a K3 surface of Picard rank one.

\begin{ACK}
We would like to thank Arend Bayer and Nick Rekuski for related discussions. 
We would like to thank Jon Woolf for letting us know the update on their article \cite{bppw}. 
We would also like to thank Laura Pertusi for sharing their work which is closely related to this article. 

K.K. is indebted to Arend Bayer and the University of Edinburgh for their kind hospitality. 
K.K. and G.O. are grateful to the Hausdorff Research Institute for Mathematics funded by 
the Deutsche Forschungsgemeinschaft (DFG, German Research Foundation) 
under Germany's Excellence Strategy – EXC-2047/1 – 390685813,
for the pleasant working conditions. 
A substantial portion of this work was carried out while G.O. was affiliated with Nagoya University.

K.K. is supported by JSPS Overseas Research Fellow and JSPS KAKENHI Grant Number 21K13780. 
N.K. is supported by the Engineering and Physical Sciences Research Council [EP/Y009959/1]. 
G.O. is supported by JSPS KAKENHI Grant Number 19K14520.

\end{ACK}
\begin{NaC}
~
\begin{itemize}
\item For a complex number $z \in \bC$, let $\Re z$ and $\Im z$ be the real part and the imaginary part of $z$ respectively.
\item We work over the complex number field $\bC$.
\item A K3 surface means an algebraic K3 surface.
\item For a variety $X$, let $D^b(X)$ denote the bounded derived category of coherent sheaves on $X$.
\item For a $\bZ$-module $M$, we put $M_\bR\coloneqq M \otimes \bR$.
\item For a smooth projective variety $X$, we define
\[ \hom(E,F)\coloneqq \dim \Hom(E,F),~ \ext^i(E,F)\coloneqq \dim \Ext^i(E,F) \]
for $E,F \in D^b(X)$ and $i\in\bZ$. 
\end{itemize}
\end{NaC}

\section{Bridgeland stability conditions}

\subsection{Slicings of triangulated categories}
In this subsection, we recall the notion of slicings of triangulated categories following \cite{bri}.

Let $\mcD$ be a triangulated category.

\begin{dfn}[{\cite[Definition 3.3]{bri}}]\label{def:slicing}
A \textit{slicing} of the triangulated category $\mcD$ is a collection $\mcP=\{\mcP(\phi) \}_{\phi \in \bR}$ of full additive subcategories of $\mcD$ such that the following conditions hold: 
\begin{enumerate}
\item For all $\phi\in\bR$, we have $\mcP(\phi+1)=\mcP(\phi)[1]$.
\item For $\phi_1>\phi_2$ and $A_i \in\mcP(\phi_i)$, we have $\Hom(A_1,A_2)= 0$.  
\item For any non-zero object $E \in \mcD$, there exist real numbers 
\[\phi_1 > \cdots >\phi_n\]
and a sequence 
\[
\xymatrix@C=0.4em@R=0.6em{
0_{\ } \ar@{=}[r] & E_0 \ar[rrrr] &&&& E_1 \ar[rrrr] \ar[dll] &&&& E_2
\ar[rr] \ar[dll] && \ldots \ar[rr] && E_{n-1}
\ar[rrrr] &&&& E_n \ar[dll] \ar@{=}[r] &  E_{\ } \\
&&& A_1 \ar@{-->}[ull] &&&& A_2 \ar@{-->}[ull] &&&&&&&& A_n \ar@{-->}[ull] 
}
\]
of exact triangles in $\mcD$ such that each $A_i$ is contained in $\mcP(\phi_k)$ for $1 \leq i \leq n$. Then we put \[\phi^+_\mcP(E)\coloneqq\phi_1,~ \phi^-_\mcP(E)\coloneqq\phi_n.\] 
This filtration is called the \textit{Harder--Narasimhan filtration of $E$}. We put $\phi^+_\mcP(E)\coloneqq\phi_1$ and $\phi^-_\mcP(E)\coloneqq\phi_n$. 
\end{enumerate}
\end{dfn}
Note that for a non-zero object $E \in \mcD$, a sequence of exact triangles in Definition \ref{def:slicing} (3) is uniquely determined up to isomorphisms. 

For a slicing $\mcP$ and an interval $I \subset \bR$, let $\mcP(I)$ denote the extension closed full subcategory of $\mcD$ generated by subcategories $\mcP(\phi)$ for $\phi \in I$.
If the length of an interval $I$ is less than one, $\mcP(I)$ is a quasi-abelian category by {\cite[Lemma 4.3]{bri}}. In particular, $\mcP(\phi)$ is a quasi-abelian category for any $\phi \in \bR$. We will freely use the terminologies in {\cite[Section 4]{bri}}.
On the other hand, as stated in {\cite[Section 3]{bri}}, $\mcP((\phi, \phi+1])$ is the heart of a bounded t-structure in $\mcD$ for any $\phi \in \bR$, in particular, $\mcP((\phi, \phi+1])$ is an abelian category.

\begin{dfn}[{\cite[Definition 5.7]{bri}}]\label{def:locally finite}
A slicing $\mcP$ of the triangulated category $\mcD$ is called \textit{locally finite} if there exists $\varepsilon > 0$ such that the category $\mcP((\phi-\varepsilon, \phi+\varepsilon))$ is a quasi-abelian category of finite length for each $\phi \in \bR$. 
\end{dfn}
The notion of finite length quasi-abelian categories is defined using strict monomorphisms and strict epimorphisms as in {\cite[Section 4]{bri}} and {\cite[Section 2.1]{bppw}}.
We denote the set of locally finite slicings of $\mcD$ by $\Slice(\mcD)$. 
For slicings $\mcP, \mcQ \in \Slice(\mcD)$, we define 
\[
d(\mcP, \mcQ)\coloneqq
\sup_{0 \neq E \in \mcD} \max\left\{
|\phi^-_{\mcP}(E)-\phi^-_\mcQ(E)|, |\phi^+_{\mcP}(E)-\phi^+_\mcQ(E)|
\right\}. 
\]
Then $(\Slice(\mcD),d)$ is a generalized metric space by {\cite[Section 6]{bri}}.

\subsection{Stability conditions on triangulated categories}\label{sec:stab}
In this subsection, we recall the notion of stability conditions following \cite{bri}.

Let $\mcD$ be a triangulated category.
Fix a finitely generated free abelian group $\Lambda$ with a surjective group homomorphism $\cl: K_0(\mcD) \to \Lambda$. 
Let $\alpha:\Aut(\mcD) \to \GL(\Lambda)$ be a group homomorphism satisfying $\cl \circ [\Phi]=\alpha(\Phi) \circ \cl$ for all $\Phi \in \Aut(\mcD)$, where $[\Phi] : K_0(\mcD) \iso K_0(\mcD)$ is the linear automorphism of $K_0(\mcD)$ induced by $\Phi$.
For a group homomorphism $Z: \Lambda \to \bC$ and an object $E \in \mcD$, denote $Z(\cl(E))$ by $Z(E)$ for simplicity.

\begin{dfn}[{\cite[Definition 5.1]{bri}}]\label{def:pre-stability}
A \textit{pre-stability condition} on the triangulated category $\mcD$ is a pair $\sigma=(Z, \mcP)$ of a group homomorphism $Z:\Lambda \to \bC$ and a slicing $\mcP$ of $\mcD$ satisfying the following condition: for any $\phi \in \bR$ and $0 \neq E \in \mcP(\phi)$, 
we have $Z(E) \in \bR_{>0}\exp(\sqrt{-1}\pi \phi)$. 
\end{dfn}

Let $\sigma=(Z, \mcP)$ be a pre-stability condition on $\mcD$. The group homomorphism $Z$ is called the \textit{central charge} of $\sigma$. For each $\phi \in \bR$, a non-zero object $E \in \mcP(\phi)$ is called a $\sigma$-\textit{semistable} object of \textit{phase} $\phi$. 
For a $\sigma$-semistable object $E \in \mcP(\phi)$, let $\phi_\sigma(E)$ denote the phase $\phi$ of $E$ with respect to $\sigma$.
By {\cite[Lemma 5.2]{bri}}, the quasi-abelian category $\mcP(\phi)$ is abelian for any $\phi \in \bR$. For $\phi \in \bR$, a simple object $E \in \mcP(\phi)$ is called a $\sigma$-\textit{stable} object of phase $\phi$. 
For $1 \leq i \leq n$, each object $A_i \in \mcP(\phi_i)$ in the Harder--Narasimhan filtration of $E$ as in  Definition \ref{def:slicing} (3) is called a $\sigma$-\textit{semistable factor} of $E$. 

\begin{dfn}[{\cite[Section 5]{bri}}]\label{def:mass}
Let $\sigma=(Z, \mcP)$ be a pre-stability condition on $\mcD$. For a non-zero object $E \in \mcD$, we define the \textit{mass} $m_\sigma(E)$ of $E$ with respect to $\sigma$ as
\[ m_\sigma(E)\coloneqq\sum_{i=1}^{n}|Z(A_i)|, \]
where $(A_1, \cdots, A_n)$ is the sequence of all $\sigma$-semistable factors of $E$.
\end{dfn}

For a pre-stability condition $\sigma=(Z,\mcP)$ on $\mcD$ and a non-zero object $E \in \mcD$, we have
\begin{equation}\label{eq:triangle inequality}
 m_\sigma(E) \geq |Z(E)|  
\end{equation}
by the triangle inequality.

We fix a norm $||-||$ on $\Lambda_\bR$. 
The definition of stability conditions on $\mcD$ is as follows:

\begin{dfn}[\cite{bri},\cite{KS08}]\label{def:support property}
A pre-stability condition $\sigma=(Z, \mcP)$ on $\mcD$ is a \textit{stability condition} on $\mcD$ if $\sigma$ satisfies the \textit{support property}, i.e., there exists a constant $C>0$ such that for any $\sigma$-semistable object $E \in \mcD$, we have 
\[ \|\cl(E)\|<C|Z(E)|. \]
\end{dfn}

Let $\sigma=(Z,\mcP)$ be a stability condition on $\mcD$.
By {\cite[Appendix B]{bm11}}, the slicing $\mcP$ is locally finite. In particular, $\mcP(\phi)$ is of finite length. Therefore, for $\phi \in \bR$ and a $\sigma$-semistable object $E \in \mcP(\phi)$, there is a \textit{Jordan--H\"older filtration}
\[ 0=E_0 \subset E_1 \subset \cdots \subset E_n=E \]
in $\mcP(\phi)$ such that the quotient $E_i/E_{i-1}$ is $\sigma$-stable of phase $\phi$ for $1 \leq i \leq n$.
Then each quotient $E_i/E_{i-1}$ is called a $\sigma$-\textit{stable factor} of $E$. 
Indeed, for a $\sigma$-semistable object $E$, $\sigma$-stable factors of $E$ together with their multiplicities do not depend on a choice of a Jordan--H\"older filtration of $E$. 

Denote the set of all stability conditions on $\mcD$ by $\Stab(\mcD)$. By definition, we have
\[ \Stab(\mcD) \subset \Hom(\Lambda, \bC) \times \Slice(\mcD). \]
For $\sigma=(Z, \mcP), \tau=(W,\mcQ) \in \Hom(\Lambda, \bC) \times \Slice(\mcD)$, we define
\begin{equation}\label{eq:metric}
d(\sigma, \tau)\coloneqq \max\left\{\|Z-W\|, d(\mcP, \mcQ)\right\},
\end{equation}
where $\|-\|$ is the operator norm on $\Hom(\Lambda, \bC)$ given by
\begin{equation}\label{eq:operator norm}
\|Z\|\coloneqq\sup\left\{|Z(v)| : v\in\Lambda_\bR \text{ with }\|v\|=1 \right\}.
\end{equation}
The set $\Hom(\Lambda, \bC) \times \Slice(\mcD)$ is a generalized metric space. In particular, $\Stab(\mcD)$ is also a generalized metric space. By {\cite[Theorem 1.2]{bri}}, the projection map $\Stab(\mcD) \to \Hom(\Lambda, \bC)$ is a local homeomorphism and $\Stab(\mcD)$ is a complex manifold. The space $\Stab(\mcD)$ is called the space of stability conditions on $\mcD$.
The group $(\bC,+)$ acts on $\Hom(\Lambda, \bC) \times \Slice(\mcD)$ as follows:
For $\lambda  \in \bC$ and $\sigma=(Z,\mcP) \in \Hom(\Lambda, \bC) \times \Slice(\mcD)$, 
\begin{equation*}
\sigma\cdot\lambda \coloneqq (\exp(-\sqrt{-1}\pi\lambda) \cdot Z, \{\tilde{\mcP}(\phi)\}_{\phi \in \bR})
\end{equation*}
where we define $\tilde{\mcP}(\phi)\coloneqq\mcP(\phi+\Re(\lambda))$ for $\phi \in \bR$.

We often use the characterization of pre-stability conditions via the hearts of bounded t-structures.

\begin{rmk}[{\cite[Proposition 5.3]{bri}}]\label{rem:slicing and heart}
Giving a pre-stability condition $(Z, \mcP)$ is equivalent to giving a pair $(Z, \mcA)$ of a group homomorphism $Z \colon \Lambda \to \bC$ and the heart $\mcA$ of a bounded t-structure satisfying the following conditions: 
    \begin{itemize}
    \item[(a)] For any non-zero object $E \in \mcA$, we have $Z(E) \in \bH \cup \bR_{<0}$.
    For a non-zero object $E \in \mcA$, an object $E$ is $Z$-semistable if for any non-zero subobject $A$ of $E$, we have $\arg Z(A) \leq \arg Z(E)$, where we regard $\arg:\bH \cup \bR_{<0} \to (0,\pi]$.
    \item[(b)] For any object $E \in \mcA$, there is a filtration 
    \[ 0=E_0 \subset E_1 \subset \cdots \subset E_n=E \]
    in $\mcA$ such that for any $1 \leq i \leq n$, the quotient object $A_i\coloneqq E_i/E_{i-1}$ is $Z$-semistable with 
    \[ \arg Z(A_1)> \cdots > \arg Z(A_n).\] 
\end{itemize}
For a pre-stability condition $(Z, \mcP)$, we have the pair $(Z,\mcA)$ satisfying the conditions (a) and (b), where we put $\mcA \coloneqq \mcP((0,1])$. Conversely, for a pair $(Z, \mcA)$ satisfying the conditions (a) and (b), we have the pre-stability condition $(Z, \mcP)$ on $\mcD$ as follows:\\ For a real number $\phi \in (0,1]$, we define
\[ \mcP(\phi) \coloneqq \left\{ E \in \mcA : \text{$E$ is $Z$-semistable with $\arg Z(E)=\phi \pi$}\right\} \cup \{0\}.  \]
A general case is defined by the property $\mcP(\phi+1)=\mcP(\phi)[1]$ for $\phi \in \bR$.
Then, by the abuse of notation, we identify the pre-stability condition $(Z, \mcP)$ with the pair $(Z, \mcA)$.
\end{rmk}

\subsection{Stability conditions and infinite dimensional projective spaces}
We keep the notations as in the previous subsection. 
In this subsection, we recall the definition of the continuous map from the space $\Stab(\mcD)$ of stability conditions on $\mcD$ to the infinite dimensional projective space following \cite{bdl20}.

Fix a set $\mcS$ of isomorphism classes of objects in $\mcD$.  By {\cite[Proposition 8.1]{bri}}, the \textit{mass map}
\[
m \colon \Stab(\mcD) \to \bR^{\mcS}_{\geq 0}, \quad 
\sigma \mapsto (m_\sigma(E))_{E \in \mcS}. 
\]
is continuous. 

We define the projective space $\bP^{\mcS}_{\geq 0}$ as the topological space
\[ \bP^\mcS_{\geq 0}\coloneqq (\bR_{\geq 0}^\mcS\setminus\{0\})/\bR_{>0},\]
where the topology on $\bP^\mcS_{\geq 0}$ is defined by 
the quotient topology of the product topology on $\bR^{\mcS}_{\geq 0} \setminus \{0\}$. 
Then $m$ induces the continuous map 
\begin{equation} \label{eq:projmass}
    \bP m \colon \Stab(\mcD)/\bC \to \bP^{\mcS}_{\geq 0}. 
\end{equation}


\subsection{Stability conditions on K3 surfaces}\label{Sec:K3 stab}
Let $X$ be a K3 surface. 
In this subsection, we recall the description of the distinguished connected component of the space of stability conditions on the derived category $D^b(X)$ of the K3 surface $X$, following \cite{bri08}.

\subsubsection{Mukai lattice and Mukai pairing}
We fix 
\[
\Lambda=\Halg(X, \bZ)\coloneqq H^0(X, \bZ) \oplus \NS(X) \oplus H^4(X, \bZ), 
\]
and the homomorphism 
\[
\cl=v \colon K(X) \to \Halg(X, \bZ), \quad [E] \mapsto \ch(E).\sqrt{\td(X)}. 
\]
For $v_1=(r_1, D_1, s_1), v_2=(r_2, D_2, s_2) \in \Halg(X, \bZ)$,
we define the \textit{Mukai pairing} of $v$ and $w$ as
\[
\langle v, w \rangle\coloneqq
D_1D_2-s_1r_2-s_2r_1
\]
The Riemann--Roch formula
\[ \langle v(E), v(F) \rangle=-\chi(E,F) \]
holds for any objects $E,F \in D^b(X)$.
The lattice $\Halg(X,\bZ)$ defined by the Mukai pairing is called the \textit{Mukai lattice} of $X$. The signature of $\Halg(X, \bZ)$ is $(2, \rho(X))$, where $\rho(X)$ is the Picard rank of $X$. 
Let us 
\[
\Delta(X)\coloneqq\left\{\delta \in \Halg(X, \bZ) : \delta^2=-2 \right\}
\]
denote the set of spherical classes on $X$.

Let $\mcP(X) \subset \Halg(X, \bZ) \otimes \bC$ be an open subset consisting of vectors whose real and imaginary parts span a positive definite plane in $\Halg(X, \bZ)_\bR$. 
Let $\mcP^+(X)$ denote the connected component of $\mcP(X)$ containing the vectors of the form $\exp(B+\sqrt{-1}\omega)$, where $B,\omega \in \NS(X)_{\bR}$ with $\omega$ ample. 
Finally, we put 
\[
\mcP^+_0(X) \coloneqq \mcP^+(X) \setminus 
\bigcup_{\delta \in \Delta(X)} \delta^{\perp}, 
\] 
where 
\[
\delta^\perp \coloneqq \left\{\Omega \in \Halg(X, \bZ) \otimes \bC : \langle\Omega, \delta \rangle=0\right\} 
\]
for a spherical class $\delta \in \Delta(X)$. 

\subsubsection{The distinguished connected component}
Let $\Stab(X)$ denote the space of stability conditions on $D^b(X)$ with respect to  $(\Halg(X, \bZ), v)$. 
For a stability condition $\sigma=(Z, \mcP)$, we define $\pi(\sigma) \in \Halg(X, \bZ) \otimes \bC$ by the formula 
\[
Z(-)=\langle\pi(\sigma), - \rangle
\]
This defines a map $\pi \colon \Stab(X) \to \Halg(X, \bZ) \otimes \bC$. 
\begin{thm}[\cite{bri08}]
There is the distinguished connected component $\Stab^*(X) \subset \Stab(X)$ which is mapped by $\pi$ onto $\mcP^+_0(X)$. Moreover, the map $\pi \colon \Stab^*(X) \to \mcP^+_0(X)$ is a covering map. 
\end{thm}

In the next section, we study the restriction of the map (\ref{eq:projmass}) to the distinguished connected component $\Stab^*(X)/\bC$.

\section{Injectivity of $\bP m$}
\subsection{Overview} \label{subsection:Strategy}
Throughout this section, let $X$ be a K3 surface, and let 
$\mcS$ be the set of isomorphism classes of spherical objects in $D^b(X)$. In this section, we prove the following theorem:

\begin{thm}[Theorem \ref{thm:inj-intro}]\label{thm:inj}
The continuous map 
\begin{equation}\label{eq:K3-projmass}
\bP m \colon \Stab^*(X)/\bC \to \bP^{\mcS}_{\geq 0}, \quad 
\sigma \mapsto (m_\sigma(E))_{E \in \mcS}. 
\end{equation} 
is injective.
\end{thm}


In the proof of Theorem \ref{thm:inj}, the following result of Huybrechts plays a key role: 

\begin{thm}[{\cite[Theorem 0.2]{huy12}}] \label{thm:Huyb}
    Let
    \[ \sigma_1=(Z_1,\mcP_1), \sigma_2=(Z_2,\mcP_2) \in \Stab^*(X)\]
    be stability conditions on $D^b(X)$ with $Z_1=Z_2$. Then
     \[\sigma_1=\sigma_2\]
     holds if and only if the following condition holds: 
   for any $E \in \mcS$ and $\phi \in \bR$, 
    $E \in \mcP_1(\phi)$ if and only if $E \in \mcP_2(\phi)$.
\end{thm}

Our argument is summarized as follows: 
Take elements 
\[ [\sigma_1], [\sigma_2] \in \Stab^*(X)/\bC\]
with $\bP m ([\sigma_1])=\bP m ([\sigma_2])$. 
\begin{enumerate}
    \item We first prove that a spherical object $E$ is $\sigma_1$-semistable if and only if it is $\sigma_2$-semistable. See Section \ref{sec:sharestability}. 
    Note that at this stage, we have not yet proved that 
    $\phi_{\sigma_1}(E)=\phi_{\sigma_2}(E)$. 
    \item Next, we prove that we can choose  representatives $\sigma_i=(Z_i,\mcP_i)$ of $[\sigma_i]$ so that $Z_1=Z_2$ holds. See Section \ref{sec:shareZ}. 
    \item Finally, in Section \ref{sec:shareP}, 
    we prove the equality 
    $\phi_{\sigma_1}(E)=\phi_{\sigma_2}(E)$ 
    for any $\sigma_i$-semistable spherical object $E$. 
    Hence we conclude the injectivity of the map (\ref{eq:K3-projmass}) by Theorem \ref{thm:Huyb}. 
\end{enumerate}

\subsection{Minimality of mass and stability}

In this subsection, we study the relation between the stability of objects and their masses. 

\begin{lem} \label{lem:massstability}
Let $\sigma=(Z, \mcP)$ be a stability condition on $D^b(X)$.
For an indecomposable object $E$ in $D^b(X)$, 
the following conditions are equivalent. 
\begin{itemize}
\item[(1)] The object $E$ is $\sigma$-semistable.
\item[(2)] The equality  $m_\sigma(E)=|Z(E)|$ holds.
\end{itemize}
\end{lem}
\begin{proof}
It is enough to prove that the condition (2) implies (1).
Take the Harder--Narasimhan filtration 
\[0=E_0 \xrightarrow{f_1}E_1 \xrightarrow{f_2} E_2 \to \cdots \to E_{n-1} \xrightarrow{f_n} E_n=E \]
of $E$ with the $\sigma$-semistable factor $A_i \in \mcP(\phi_i)$, where $\phi_1> \cdots > \phi_n$.
By the assumption, we have 
\[ |Z(E)|=m_\sigma(E)=\sum_{i=1}^{n}|Z(A_i)|. \]
Therefore, for any $i<j$, we have $\phi_i-\phi_j \in 2\bZ_{>0}$.   
Consider the exact triangle
\begin{equation}\label{split triangle}
     E_{n-1} \xrightarrow{f_n} E \to A_n \to E_{n-1}[1].
     \end{equation}
By Serre duality, we obtain isomorphisms 
\begin{eqnarray*}
\Hom(A_n, E_{n-1}[1]) &\simeq & \Hom(E_{n-1}[1], A_n[2])^*\\
& \simeq & \Hom(E_{n-1}, A_n[1])^*.
\end{eqnarray*}
Since $\phi_{n-1}-\phi_n$ is a positive even number, we have $\Hom(E_{n-1}, A_n[1])=0$.
In particular, the exact triangle (\ref{split triangle}) splits. Since $E$ is indecomposable, we have $E \simeq A_n$.
\end{proof}

We recall the following non-emptiness result:

\begin{lem} \label{lem:sphexist}
Let $v \in \Delta(X)$ be a spherical class. 
For $\sigma \in \Stab^*(X)$, there is a $\sigma$-semistable spherical object $E$ with $v(E)=v$.   
\end{lem}
\begin{proof}
Take a $v$-generic stability condition $\sigma_+$ near $\sigma$. By {\cite[Corollary 6.9]{bm14}}, which ensures 
the non-emptiness of the moduli space of $\sigma_+$-stable objects, there is a $\sigma_+$-stable object $E$ with $v(E)=v$. Then $E$ is a $\sigma$-semistable spherical object with $v(E)=v$.     
\end{proof}

\begin{cor} \label{cor:massmin}
    Given a stability condition $\sigma=(Z, \mcP) \in \Stab^*(X)$ 
    and a spherical class $v \in \Delta(X)$, 
    take a $\sigma$-semistable object $E$ with $v(E)=v$. 
    Then the following equations hold: 
    \[
    m_\sigma(E)=|Z(v)|=\min_{F \in \mcS_v} m_\sigma(F), 
    \]
    where $\mcS_v$ is the set of isomorphism classes of spherical objects with $v(F)=v$. 
\end{cor}
\begin{proof}
 By definition, we have $m_\sigma(E)=|Z(v)|$. 
Moreover, by the inequality (\ref{eq:triangle inequality}), we have 
    $|Z(v)| \leq m_\sigma(F)$ for any spherical object $F$ with $v(F)=v$. 
   Hence, the assertion holds. 
\end{proof}

\begin{cor} \label{cor:abscentral}
    Take elements $\sigma_1=(Z_1,\mcP_1), \sigma_2=(Z_2,\mcP_2) \in \Stab^*(X)$. Suppose that 
    \[ m (\sigma_1)=m (\sigma_2)\]
    holds. 
    Then we have 
    \[|Z_1(v)|=|Z_2(v)|\]
    for any spherical class $v \in \Delta(X)$. 
\end{cor}
\begin{proof}
We have 
\[
|Z_1(v)|=
\min_{F\in \mcS_v}m_{\sigma_1}(F) 
=\min_{F \in \mcS_v} m_{\sigma_2}(F) 
=|Z_2(E)|, 
\]
where the first and the third equality follow from Corollary \ref{cor:massmin}, 
and the second equality follows from the assumption $m(\sigma_1)=m(\sigma_2)$. 
\end{proof}

\subsection{Matching semistable objects}
\label{sec:sharestability}


\begin{prop} \label{prop:sharestability}
    Take elements $\sigma_1=(Z_1,\mcP_1), \sigma_2=(Z_2,\mcP_2) \in \Stab^*(X)$. Suppose that we have
    \[\bP m ([\sigma_1])=\bP m ([\sigma_2]).\]
   Then a spherical object $E$ is $\sigma_1$-semistable if and only if $E$ is $\sigma_2$-semistable.
\end{prop}
\begin{proof}
    By the assumption, there exists a constant $c >0$ such that 
    \[ m_{\sigma_1}(E)=cm_{\sigma_2}(E)\]
for any spherical object $E$. 
    By replacing a representative of $[\sigma_2]$ if necessary, we may assume that $c=1$. 
    Take a $\sigma_1$-semistable spherical object $E$ with $v\coloneqq v(E)$. 
    Then we have 
    \begin{align*}
        m_{\sigma_2}(E)
        &=m_{\sigma_1}(E) \\
        &=\min_{F \in \mcS_v} m_{\sigma_1}(F) \\
        &=\min_{F \in \mcS_v} m_{\sigma_2}(F) \\
        &=|Z_2(E)|. 
    \end{align*}
    Indeed, the first and the third equality hold by our choice of the representatives $\sigma_1$ and $\sigma_2$, 
    and the second and the last equality hold by Corollary \ref{cor:massmin}.  
    By Lemma \ref{lem:massstability}, we conclude that $E$ is $\sigma_2$-semistable. 
\end{proof}

\subsection{Matching central charges}
\label{sec:shareZ}

We begin with the following easy lemma: 

\begin{lem} \label{lem:goodbasis}
    There exists a $\bR$-basis $\{v_1, \cdots, v_{\rho(X)+2}\}$ of $\widetilde{H}^{1, 1}(X, \bZ)_\bR$ satisfying the following conditions\textup{:}
    \begin{enumerate}
        \item  $v_i \in \Delta(X)$ holds for  $1 \leq i \leq \rho(X)+2$, 
        \item $\langle v_i, v_j \rangle \neq 0$ for each $i \neq j$. 
        \end{enumerate}
\end{lem}
\begin{proof}
Recall that the signature of $\NS(X)$ is $(1,\rho(X)-1)$.
Then we can take  $C, D_1, \cdots, D_{\rho(X)-1} \in \NS(X)$
such that the following conditions hold:
\begin{itemize}
\item $\{C, D_1, \cdots, D_{\rho(X)-1}\}$ forms a $\bR$-basis of $\NS(X)_\bR$, 
\item $C\cdot D_i=D_i\cdot D_j=0$ for any $i \neq j$, 
\item $C^2 >0$, 
\item $D_i^2 <-4$ for any $1 \leq i \leq \rho(X)+2$.
\end{itemize}

Now, it is straightforward to check that the Mukai vectors 
    \[ v_1 \coloneqq v(\mcO(D_1)),~ \cdots,~ v_{\rho(X)-1} \coloneqq v(\mcO(D_{\rho(X)-1}))\]
    \[ v_{\rho(X)} \coloneqq v(\mcO_X),~ 
    v_{\rho(X)+1}\coloneqq v(\mcO_X(-C)),~     
    v_{\rho(X)+2}\coloneqq v(\mcO_X(2C))\]
satisfy the required properties. 
\end{proof}
Fix a $\bR$-basis $\{v_1, \cdots, v_{\rho(X)+2} \}$ of $\widetilde{H}^{1, 1}(X, \bZ)_\bR$ as in Lemma \ref{lem:goodbasis}.

\begin{lem} \label{lem:goodlift}
Take elements $\xi_1,\xi_2 \in \Stab^*(X)/\bC$ satisfying
    \[ \bP m (\xi_1)=\bP m (\xi_2).\] 
There exist stability conditions $\sigma_1=(Z_1,\mcP_1), \sigma_2=(Z_2,\mcP_2) \in \Stab^*(X)$ satisfying $\xi_1=[\sigma_1],\xi_2=[\sigma_2]$ and the following conditions\textup{:} 
    \begin{enumerate}
        \item $Z_1(v_1)=Z_2(v_1)=1$, 
        \item $\phi_{\sigma_1}(E_1)=\phi_{\sigma_2}(E_1)$ holds for a $\sigma_1$-semistable spherical object $E_1$ with $v(E_1)=v_1$, 
        \item $m_{\sigma_1}(E)=m_{\sigma_2}(E)$ for any spherical object $E \in \mcS$, equivalently, we have $m(\sigma_1)=m(\sigma_2)$. 
    \end{enumerate}
\end{lem}
\begin{proof}
    Take any representatives $\sigma_1$ and $\sigma_2$ of $\xi_1$ and $\xi_2$ respectively. 
    Let $E_1$ be a $\sigma_1$-semistable object with $v(E_1)=v_1$. 
    Then it is also $\sigma_2$-semistable by Proposition \ref{prop:sharestability}. 
    Hence, by replacing the representatives $\sigma_1, \sigma_2$ if necessary, 
    we may assume that the conditions (1) and (2) are satisfied. 
    Note that the $\bC$-action does not change the collection of semistable objects, hence $E_1$ remains semistable with respect $\sigma_1$ and $\sigma_2$. 

    Since we assume $\bP m ([\sigma_1])=\bP m ([\sigma_2])$, 
    there exists a constant $c>0$ such that 
    \[ m_{\sigma_1}(E)=c m_{\sigma_2}(E)\]
    for any spherical object $E \in \mcS$. 
    By the condition (1), we have 
    \[
    m_{\sigma_2}(E_1)=
    |Z_2(v_1)|
    =|Z_1(v_1)|
    =m_{\sigma_1}(E_1)=c m_{\sigma_2}(E_1)
    \]
    Hence we have $c=1$ and the condition (3) also holds. 
\end{proof}


\begin{prop} \label{prop:equalcentral}
Take elements $\sigma_1=(Z_1,\mcP_1), \sigma_2=(Z_2,\mcP_2) \in \Stab^*(X)$ satisfying the conditions (1), (2) and (3) in Lemma \ref{lem:goodlift}. 
    Then we have \[ Z_1=Z_2.\]
\end{prop}


\begin{proof}
Take a $\bR$-basis $\{v_1, \cdots, v_{\rho(X)+2}\}$ of $\Halg(X,\bZ)_\bR$ as in Lemma \ref{lem:goodbasis}. 
Let $\{ v_1^*, \cdots, v_{\rho(X)+2}^*\}$ be its dual basis of  $\Halg(X, \bZ)^*_\bR$. 
Then there are real numbers $a_i, b_i, c_i, d_i \in \bR~ (1 \leq i \leq \rho(X)+2)$ such that we have
\begin{align*}
    Z_1=\sum_{i=1}^{\rho(X)+2}(a_i+\sqrt{-1}b_i)v^*_i, \\
    Z_2=\sum_{i=1}^{\rho(X)+2}(c_i+\sqrt{-1}d_i)v^*_i. 
\end{align*}
We need to show that $a_i=c_i$ and $b_i=d_i$ for  
$1 \leq i \leq \rho(X)+2$. 

Recall that $v_i \in \Delta(X)$ holds for $1 \leq i \leq \rho(X)+2$ by the condition (1) in Lemma \ref{lem:goodbasis}. By the condition (3) in Lemma \ref{lem:goodlift} and Corollary \ref{cor:abscentral}, 
for $1 \leq i \leq \rho(X)+2$, we have 
$|Z_1(v_i)|=|Z_2(v_i)|$, which is equivalent to 
\begin{equation} \label{eq:a^2+b^2}
    a_i^2+b_i^2=c_i^2+d_i^2. 
\end{equation}

Next, we consider spherical classes 
\[ w_{ij}= v_j+\langle v_i, v_j \rangle v_i \in \Delta(X)\]
for all $i \neq j$.
Note that $\langle v_i, v_j \rangle \neq 0$
by the condition (2) in Lemma \ref{lem:goodbasis}. 
By Corollary \ref{cor:abscentral}, we also have 
\begin{equation} \label{eq:wij}
    |Z_1(w_{ij})|=|Z_2(w_{ij})|. 
\end{equation}
Combining equations (\ref{eq:a^2+b^2}) and (\ref{eq:wij}), 
we obtain 
\begin{equation} \label{eq:aiaj}
    a_ia_j+b_ib_j=c_ic_j+d_id_j, \quad 
1 \leq i < j \leq \rho(X)+2 
\end{equation}
by the straightforward computation.

Moreover, the condition (1) in Lemma \ref{lem:goodlift}
implies that \[a_1=c_1=1,~b_1=d_1=0.\] 
Putting $i=1$,  we obtain
\[ a_j=c_j,\quad2 \leq j \leq \rho(X)+2\]
by the equations (\ref{eq:aiaj}). 
Then, by the equations (\ref{eq:a^2+b^2}) and (\ref{eq:aiaj}), we have
\[b_ib_j=d_id_j,\quad 1 \leq i,j \leq \rho(X)+2.\]
Hence, we obtain 
\[
(b_1, \cdots, b_{\rho(X)+2})=\pm(d_1, \cdots, d_{\rho(X)+2}), 
\]
which implies $Z_1=Z_2$ or 
$Z_1=\overline{Z}_2$.
However, since we have $Z_1, Z_2 \in \mcP^+_0(X)$, we cannot have $Z_1=\overline{Z}_2$. 
We conclude that $Z_1=Z_2$. 
\end{proof}

\subsection{Matching phases}
\label{sec:shareP}

First, we prove the following key lemma: 

\begin{lem} \label{lem:keyphase}
Take elements $\sigma_1=(Z_1, \mcP_1), \sigma_2=(Z_2,\mcP_2) \in \Stab^*(X)$ satisfying
the conditions (1), (2) and (3) in Lemma \ref{lem:goodlift}.
Let $E$ and $F$ be $\sigma_1$-semistable spherical objects satisfying the following conditions\textup{:} 
    \begin{enumerate}
        \item $\phi_{\sigma_1}(E)=\phi_{\sigma_2}(E)$, 
        \item $\langle v(E), v(F) \rangle \neq 0$, 
        \item $Z_1(F) \notin \bR \cdot Z_1(E)$. 
    \end{enumerate}

    Then we have $\phi_{\sigma_1}(F)=\phi_{\sigma_2}(F)$. 
\end{lem}
\begin{proof}
    By Proposition \ref{prop:equalcentral}, 
     we have $Z_1=Z_2$. 
    Hence, it is enough to show that 
    $|\phi_{\sigma_1}(F)-\phi_{\sigma_2}(F)| <2$. 

    First, we note that $E$ and $F$ are also $\sigma_2$-semistable by Proposition \ref{prop:sharestability}. 
    By the assumption (2), there exists an integer $l$ satisfying
    \[
    \hom(F[l], E[2])=\hom(E, F[l]) \neq 0, 
    \]
    where the first equality is deduced from Serre duality.
    Since both $E$ and $F$ are $\sigma_i$-semistable, we have 
    the following inequalities for $i=1, 2$: 
    \begin{equation} \label{eq:phaseineq}
        \phi_{\sigma_i}(E) \leq \phi_{\sigma_i}(F)+l \leq 
        \phi_{\sigma_i}(E)+2. 
    \end{equation}
    By the assumption (3), the inequalities in (\ref{eq:phaseineq}) are strict. 
    Finally, the assumption (1) implies 
    $|\phi_{\sigma_1}(F)-\phi_{\sigma_2}(F)| <2$ as required.  
\end{proof}

\begin{prop}\label{prop:sharephase}
Take elements $\sigma_1=(Z_1, \mcP_1), \sigma_2=(Z_2,\mcP_2) \in \Stab^*(X)$ satisfying
the conditions (1), (2) and (3) in Lemma \ref{lem:goodlift}.
Then, for any $E \in \mcS$ and $\phi \in \bR$, $E \in \mcP_{\sigma_1}(\phi)$ if and only if $E \in \mcP_{\sigma_2}(\phi)$. 
\end{prop}
\begin{proof}
 By Proposition \ref{prop:equalcentral}, 
    we have $Z_1=Z_2$ again. 
By Lemma \ref{lem:sphexist}, we may take a $\sigma_1$-semistable object $E_i$ with $v(E_i)=v_i$ for $1 \leq i \leq \rho(X)+2$.
By Proposition \ref{prop:sharestability}, 
$E_i$ is $\sigma_2$-semistable for each $i$. 
We first prove that 
\[ \phi_{\sigma_1}(E_i)=\phi_{\sigma_2}(E_i)\]
holds for  $1 \leq i \leq \rho(X)+2$. 
By the condition (2) in Lemma \ref{lem:goodlift}, 
we have 
\[\phi_{\sigma_1}(E_1)=\phi_{\sigma_2}(E_1).\] 
Since $\pi(\sigma_1)$ is contained in $\mcP^+(X)$, the dimension of the image of the $\bR$-linear map 
$Z_1 \otimes \bR \colon \Halg(X,\bZ)_\bR \to \bC$ is two. 
Therefore, we have
 \[Z_1(E_{i_0}) \notin \bR \cdot Z_1(E_1), \]
 where $i_0$ is some integer satisfying $2 \leq i_0 \leq \rho(X)+2$.
Moreover, $\langle v_1, v_{i_0} \rangle \neq 0$ holds by 
the condition (2) in Lemma \ref{lem:goodbasis}. 
Hence, we can apply Lemma \ref{lem:keyphase} to obtain 
\[ \phi_{\sigma_1}(E_{i_0})=\phi_{\sigma_2}(E_{i_0}).\] 
Take an arbitrary integer 
$i \in \{1, 2 \cdots, \rho(X)+2 \}$. 
Then we have either 
$Z_1(E_{i}) \notin \bR \cdot Z_1(E_1)$ 
or 
$Z_1(E_{i}) \notin \bR \cdot Z_1(E_{i_0})$. 
In either case, we can apply Lemma \ref{lem:keyphase} as in the previous paragraph and conclude that 
\[ \phi_{\sigma_1}(E_i)=\phi_{\sigma_2}(E_i)\]
as claimed.

We consider a general case. Take a $\sigma_1$-semistable spherical object $E$ with $v(E)=v$. 
If we have 
\[ Z_1(E) \notin \bR \cdot Z_1(E_i)~ \text{and} ~\langle v, v_i \rangle \neq 0\] for some $i \in \{1, 2 \cdots, \rho(X)+2\}$,  
then we immediately get
\[ \phi_{\sigma_1}(E)=\phi_{\sigma_2}(E)\]
by Lemma \ref{lem:keyphase}. 
Suppose that 
\begin{equation}\label{eq:or}
Z_1(E) \in \bR \cdot Z_1(E_i) 
~ \text{or} ~ \langle v, v_i \rangle = 0
\end{equation}
holds for all $1 \leq i \leq \rho(X)+2$. 
Since $\{v_1, \cdots, v_{\rho(X)+2}\}$ is a $\bR$-basis of  $\Halg(X, \bZ)_\bR$, 
we obtain $\langle v, v_j \rangle \neq 0$
for some $j \in \{1, 2, \cdots, \rho(X)+2 \}$.
Then we have 
\begin{equation}\label{eq:Ej}
Z_1(E) \in \bR \cdot Z_1(E_j)
\end{equation}
by the assumption (\ref{eq:or}). 
Since the dimension of the image of $Z_1\otimes \bR$ is two, we have
 \begin{equation}\label{eq:Ek}
      Z_1(E) \notin \bR \cdot Z_1(E_k)
 \end{equation}
for some $k \in \{1,2, \cdots, \rho(X)+2 \} \setminus \{j\}$. 
By the assumption (\ref{eq:or}), we must have $\langle v, v_k \rangle = 0$.
Consider the spherical class 
$w \coloneqq v_k+\langle v_j, v_k \rangle v_j$.
By Lemma \ref{lem:sphexist}, we can take a $\sigma_1$-semistable spherical object $F$ with $v(F)=w$. 
 Recall that $\langle v_j, v_k \rangle \neq 0$ holds by the condition (2) in Lemma \ref{lem:goodbasis}.  
 Since 
 \begin{equation}\label{eq:ZE}
 Z_1(F)=Z_1(E_k)+\langle v_j, v_k \rangle Z_1(E_j)
 \end{equation}
 holds, we obtain 
 $Z_1(F) \notin \bR \cdot Z_1(E_j)$ 
 by (\ref{eq:Ej}) and (\ref{eq:Ek}).
Moreover, we have 
\[ \langle w, v_j \rangle=-\langle  v_j, v_k \rangle \neq 0 \]
by the condition (2) in Lemma \ref{lem:goodbasis}.
Hence, we can apply Lemma \ref{lem:keyphase} to the spherical objects 
$F$ and $E_j$ and obtain $\phi_{\sigma_1}(F)=\phi_{\sigma_2}(F)$. 
By (\ref{eq:Ej}), (\ref{eq:Ek}) and (\ref{eq:ZE}), we have 
\[ Z_1(E) \notin \bR \cdot Z_1(F).\]
Note that $\langle v, w \rangle = \langle v_j, v_k \rangle \langle v, v_j \rangle  \neq 0$ 
holds by the condition (2) in Lemma \ref{lem:goodbasis} and the choice of $j$.
Hence, we can apply Lemma \ref{lem:keyphase} to the spherical objects $F$ and $E$, and get $\phi_{\sigma_1}(E)=\phi_{\sigma_2}(E)$.
\end{proof}

Summarizing the above arguments, we have the proof of Theorem \ref{thm:inj}. 

\begin{proof}[Proof of Theorem \ref{thm:inj}]
Take elements $\xi_1,\xi_2 \in \Stab^*(X)/\bC$ 
  such that
    \[ \bP m (\xi_1)=\bP m (\xi_2)\]
    holds. Take stability conditions $\sigma_1, \sigma_2 \in \Stab^*(X)$ satisfying 
    \[\xi_1=[\sigma_1],~ \xi_2=[\sigma_2]\]
    as in Lemma \ref{lem:goodlift}. 
    By Propositions \ref{prop:equalcentral} and \ref{prop:sharephase}, the assumptions in Theorem \ref{thm:Huyb} are satisfied for $\sigma_1$ and $\sigma_2$. 
    Hence, we get $\sigma_1=\sigma_2$ as required. 
\end{proof}



\section{Certain points on the boundary via spherical objects}
Let $X$ be a K3 surface. 
In this section, we consider certain points on the boundary of the image of the mass map 
$\bP m \colon \Stab^*(X)/\bC \to \bP^{\mcS}_{\geq 0}$ 
as in the previous section. 

\subsection{Hom-functionals}\label{sec:Hom-functional}
Following \cite{bdl20}, we introduce a boundary point induced by the hom-functional. 
\begin{dfn}[{\cite[Definition 4.8]{bdl20}}]
For a spherical object $A \in \mcS$, 
we define the function 
$\overline{\hom}(A, -) \colon \mcS \to \bR_{\ge0}$ as follows: 
\[
\overline{\hom}(A, E) \coloneqq 
\begin{cases}
\sum_{i \in \bZ} \ext^i(A,E) & (E\neq A[i]\text{ for any }i\in\bZ), \\
0 & (\text{otherwise}). 
\end{cases}
\]
The corresponding element in $\bP^{\mcS}_{\geq 0}$ is denoted by
\[
h_A \coloneqq \left[(\overline{\hom}(A, E))_{E \in \mcS}\right]
\in \bP^{\mcS}_{\geq 0}.
\]
\end{dfn}

Recall that the autoequivalence group $\Aut(D^b(X))$ acts on $\Stab(D^b(X))$ by {\cite[Lemma 8.2]{bri}}.

\begin{thm}[{\cite[Theorem 4.9, Corollary 4.13]{bdl20}}]
\label{thm-Hom-functional}
For a spherical object $A \in \mcS$ such that
$A$ is $\sigma$-semistable for some stability condition $\sigma$ on $D^b(X)$, 
we have
\[
\lim_{n\to\pm\infty}\frac{m_\sigma(T_A^n E)}{n}
=m_\sigma(A)\cdot \overline{\hom}(A, E)
\text{ for any spherical object }E\in\mcS. 
\]
In particular, if the spherical twist $T_A$ along $A$ preserves $\Stab^*(X)$, the element $h_A$ lies on the boundary of the closure of $\bP m (\Stab^*(X)/\bC)$.
\end{thm}

\begin{ex}
The following objects in $D^b(X)$ satisfy the conditions in Theorem \ref{thm-Hom-functional}: 
\begin{enumerate}
\item 
A spherical vector bundle which is Gieseker-stable with respect to an ample divisor class 
(\cite[Proposition 7.5, Proposition 7.6]{Har}). 

\item 
A spherical torsion sheaf $\mcO_C(k)$ for a $(-2)$-curve $C$ and $k\in\bZ$
(\cite[Proposition 7.7]{Har}). 


\item
Any spherical object when $X$ is of Picard rank one (\cite[Theorem 1.3, Corollary 6.9]{BB}). 

\end{enumerate}
\end{ex}

\subsection{Lax stability conditions on K3 surfaces}\label{sec:Lax stabilty on K3s}

In this subsection, 
we construct examples of points in the boundary of the image $\bP m(\Stab^*(X)/\bC)$ using lax stability conditions on $D^b(X)$. 

\subsubsection{Lax stability conditions on triangulated categories}
We recall the definition of lax stability conditions on a triangulated category $\mcD$ following \cite{bppw}.
We use the same notations as in Section \ref{sec:stab}.

\begin{dfn}[{\cite[Definition 4.2]{bppw}}]
A \textit{lax pre-stability condition} on the triangulated category $\mcD$ is a pair $\sigma=(Z, \mcP)$ of a group homomorphism $Z:\Lambda \to \bC$ and a locally finite slicing $\mcP \in \Slice(\mcD)$ of $\mcD$ satisfying the following condition: for any $\phi \in \bR$ and $0 \neq E \in \mcP(\phi)$, 
we have $Z(E) \in \bR_{\geq 0}\exp(\sqrt{-1}\pi \phi)$. 
\end{dfn}


For a lax pre-stability condition $\sigma$ on $\mcD$, 
we can define the notions of $\sigma$-(semi)stability, phases, $\sigma$-semistable factors, and masses of objects in $D^b(X)$, see Definition \ref{def:mass} and the texts below Definition \ref{def:pre-stability}. 
Note that an object $E$ in a quasi-abelian category $\mcA$ is called \textit{simple} if $E$ does not have any non-trivial \textit{strict} subobjects in $\mcA$. 
 
\begin{rmk}
Let $\sigma=(Z,\mcP)$ be a lax pre-stability condition  on $\mcD$. Take a real number $\phi \in \bR$ and a $\sigma$-semistable object $E \in \mcP(\phi)$.
Recall that we assume the slicing $\mcP$ to be locally finite. In particular,
the quasi-abelian category $\mcP(\phi)$ is of finite length. Therefore, there is a strict filtration
\[ 0=E_0 \subset E_1 \subset \cdots \subset E_n=E \]
in the quasi-abelian category $\mcP(\phi)$ such that the strict quotient $E_i/E_{i-1}$ is $\sigma$-stable of phase $\phi$ for $1 \leq i \leq n$.
Since $\mcP(\phi)$ is not necessarily abelian, we do not know the uniqueness property of $\sigma$-stable factors as in Jordan--H\"older filtrations.
\end{rmk}

\begin{dfn}[{\cite[Definition 4.3]{bppw}}]\label{def:massless}
Let $\sigma$ be a lax pre-stability condition on $\mcD$.
An object $E$ of $\mcD$ is called \textit{massive} with respect to $\sigma$ if we have $m_\sigma(E) \neq 0$, and \textit{massless} with respect to $\sigma$ if we have $m_\sigma(E)=0$. 
\end{dfn}

\begin{dfn}\label{def:lax support}
A lax pre-stability condition $\sigma=(Z, \mcP)$ is called a \textit{lax stability condition} if it satisfies the \textit{support property}, i.e., 
if there exists a constant $C > 0$ such that for any massive $\sigma$-stable object $E$, we have 
\[
\|\cl(E)\|<C|Z(E)|.
\]
\end{dfn}

\begin{rmk}
After writing the first draft of this article, Jon Woolf informed the authors that they now require a stronger support property for lax pre-stability conditions, in order to prove a deformation result. 
The new version is called the \textit{$\epsilon$-lax support property}, see Definition \ref{def:delta-lax support}. 

Although we do not use either versions of support properties in the current article, we will prove them in our setting in Propositions \ref{prop:support} and \ref{prop:delta-lax support property} for the sake of completeness. 
\end{rmk}

Denote the set of lax stability conditions on $\mcD$ by $\Stab^l(\mcD)$. We define 
\[ \Stab^L(\mcD)\coloneqq \Stab^l(\mcD) \cap \overline{\Stab(\mcD)}, \]
where $\overline{\Stab(\mcD)}$ is the closure of $\Stab(\mcD)$ in $\Hom(\Lambda, \bC) \times \Slice(\mcD)$. 
Note that $\Stab^l(\mcD)$ and $\Stab^L(\mcD)$ are also generalized metric spaces.
We have the following inclusions: 
\[
\Stab(\mcD) \subset \Stab^L(\mcD) \subset \Stab^l(\mcD). 
\]

\begin{rmk}\label{rem:lax-0}
For a slicing $\mcP \in \Slice(\mcD)$, the pair $(0, \mcP)$ always defines a lax pre-stability condition on $\mcD$. Since there does not exist massive $\sigma$-stable objects, we have $(0, \mcP) \in \Stab^l(\mcD)$.
\end{rmk}

\begin{rmk}
For a stability condition $\sigma=(Z,\mcP) \in \Stab(\mcD)$, we have
\[ \lim_{t \to +\infty}\sigma \cdot (-t\sqrt{-1})=(0, \mcP) \in \Stab^L(\mcD) \]
by Remark \ref{rem:lax-0}.
\end{rmk}

Let $\Stab^L(\mcD)^\star \subset \Stab^L(\mcD)$ be the subset of lax stability conditions whose central charges are non-zero. 
The action of the group $(\bC,+)$ on $\Hom(\Lambda, \bC)\times \Slice(\mcD)$ preserves $\Stab^l(\mcD)$,  $\Stab^L(\mcD)$ and $\Stab^L(\mcD)^\star$. 
By \cite[Section 13.2]{bppw}, the mass map (\ref{eq:projmass}) extends to a continuous map 
\[
\bP m \colon \Stab^L(\mcD)^\star/\bC \to \bP^\mcS_{\geq 0}.
\]
Since $\Stab(\mcD)/\bC$ is dense in $\Stab^L(\mcD)/\bC$ and $\bP m$ is continuous, we have 
\[
\bP m (\Stab^L(\mcD)^\star/\bC) \subset \overline{\bP m (\Stab(\mcD)/\bC)}. 
\]




\subsubsection{Construction of stability conditions on K3 surfaces}
Let $X$ be a K3 surface. 
In this subsection, we recall the construction of stability conditions on $D^b(X)$ following \cite{bri08}.
The construction is based on Remark \ref{rem:slicing and heart}.

Take $B, \omega \in \NS(X)_\bR$ with $\omega$ ample. 
For a coherent sheaf $E \in \Coh(X)$, we define the \textit{slope} $\mu_\omega(E)$ of $E$ as follows: 
\[
\mu_\omega(E)\coloneqq\begin{cases}
\frac{\omega \cdot \ch_1(E)}{\ch_0(E)} & (\ch_0(E) > 0), \\
+\infty & (\ch_0(E)=0). 
\end{cases}
\]

A coherent sheaf $E$ on $X$ is called \textit{$\mu_\omega$-semistable} (resp. $\mu_\omega$-stable)
if we have 
 \[   \mu_\omega(F) \leq \mu_\omega(G)~(\text{resp.}~ \mu_\omega(F) <\mu_\omega(G))
\]
for any short exact sequence 
\[
0 \to F \to E \to G \to 0.
\]

We define full subcategories $\mcT$ and  $\mcF$ of $\Coh(X)$ as follows: 
\begin{align*}
&\mcT \coloneqq \langle 
E \in \Coh(X) \colon E \text{ is $\mu_\omega$-semistable with } \mu_\omega(E) > B\omega
\rangle, \\
&\mcF \coloneqq \langle
E \in \Coh(X) \colon E \text{ is $\mu_\omega$-semistable with } \mu_\omega(E) \leq B\omega
\rangle,
\end{align*}
where, $\langle-\rangle$ denotes the extension closure. 
As shown in \cite[Lemma 6.1]{bri08}, the pair $(\mcT, \mcF)$ forms a torsion pair on $\Coh(X)$
in the sense of \cite{hrs96}. 
Hence, by the general theory of \cite{hrs96}, 
the category 
\[
\mcA \coloneqq \langle \mcF[1], \mcT \rangle \subset D^b(X)
\]
is the heart of a bounded t-structure of $D^b(X)$. 
We define a group homomorphism $Z_{B, \omega} \colon \Halg(X,\bZ) \to \bC$ as follows: 
\[
Z_{B, \omega}(v) \coloneqq \langle 
\exp(B+\sqrt{-1}\omega), v
\rangle
\]
for $v \in \Halg(X,\bZ)$.
In the sense of Remark \ref{rem:slicing and heart}, we can construct stability conditions on $D^b(X)$ as follows:

\begin{thm}[{\cite[Lemma 6.2 and Proposition 11.2]{bri08}}]
\label{thm:non-existence of certain spherical}
Suppose that we have 
\[ Z_{B, \omega}(E) \notin \bR_{\leq 0}\] for any spherical sheave $E$ on $X$. Then we have  
$(Z_{B, \omega}, \mcA(B,\omega)) \in \Stab^*(X)$. 
\end{thm}

\subsubsection{Construction of lax stability conditions on K3 surfaces} \label{sec:constructlax}
Let $X$ be a K3 surface. In this subsection, we construct a lax stability condition on $D^b(X)$ as a degeneration of stability conditions constructed in Theorem \ref{thm:non-existence of certain spherical}.

Fix an ample divisor $H \in \NS(X)$ with $H^2=2d$.
By the abuse of notation, for $\delta=(r, D, s) \in \Halg(X,\bZ)$ with $r \neq 0$, we denote
\[ \mu_H(\delta)\coloneqq \frac{H D}{r} \in \bQ. \]
Fix a rational number $\mu \in \bQ$. We set
\[
\Delta_\mu(X) \coloneqq \left\{\delta\in\Delta(X) : \mu_H(\delta)=\mu\right\},
\]
and
\[
\Delta_\mu^+(X)\coloneqq \left\{(r,D,n)\in\Delta_\mu(X) : r>0\right\}. 
\]
In the following, we always assume that $\Delta_\mu^+(X)$ is non-empty. 
We put 
\[ r_0 \coloneqq \min\{r \in \bZ_{>0} : (r,D,s) \in \Delta^+_\mu(X)  \}.\]

\begin{rmk} \label{rmk:Delta+}
\begin{enumerate}
\item Assume that $\NS(X)=\bZ H$ holds. Then the set $\Delta^+_{\mu}(X)$ is either empty or a singleton. 
Suppose that $\Delta^+_\mu(X)$ is non-empty. 
Take an element $\delta=(r,aH,s) \in \Delta^+_\mu(X)$ and
put $k \coloneqq \gcd(r,a)>0$. Then there exist integers $r_1$ and $a_1$ such that $r=kr_1,~a=ka_1$.
Since $\delta$ is a spherical class, we have
\[ k(2dka^2_1-2r_1s)=-2. \]
Since $2dka^2_1-2r_1s$ is even, we have $k=1$. Therefore, since $\mu_H(\delta)=\mu$ holds, we have $\Delta^+_\mu(X)=\{\delta\}$. 

\item Assume now that $\rho(X)>1$ and consider the case $\mu=0$. In this case, $r_0=1$, as we have $(1, 0, 1) \in \Delta^+_{\mu}(X)$. 
We claim that there are infinitely many spherical classes $\delta \in \Delta_{\mu=0}^+(X)$. 
The set $H^\perp \subset \NS(X)$ of divisors orthogonal to $H$ is a free $\bZ$-module of rank $\rho(X)-1$. In particular, it is an infinite set. 
Since we have $(1, D, (D^2+2)/2) \in \Delta^+_{\mu=0}(X)$ for any $D \in H^\perp$, the claim follows. 
\end{enumerate}
\end{rmk}

We fix a spherical class $\delta_0=(r_0, D_0, s_0) \in \Delta^+_\mu(X)$.
By the non-emptiness result of the moduli space of $H$-semistable sheaves of Mukai vector $\delta_0$, there is a $\mu_H$-semistable vector bundle $A$ on $X$ with $v(A)=\delta_0$. By the minimality of $r_0$, the vector bundle $A$ is $\mu_H$-stable. Therefore, such a vector bundle is unique up to isomorphism.

In this subsection, we choose $B_0$ as 
\[ B_0 \coloneqq  \frac{1}{r_0}D_0 \in \NS(X)_\bQ. \]
By definition, we have 
\begin{equation}\label{eq:slope of delta}
B_0H=\mu_H(A)=\mu.
\end{equation}

As in the previous subsection, we define the heart of a bounded t-structure  
\[\mcA_\mu \coloneqq \langle \mcF_\mu[1], \mcT_\mu \rangle \] 
on $D^b(X)$, where we put
\begin{align*}
&\mcT_\mu \coloneqq \langle 
E \in \Coh(X) \colon E \text{ is $\mu_H$-semistable with } \mu_H(E) > \mu
\rangle, \\
&\mcF_\mu \coloneqq \langle
E \in \Coh(X) \colon E \text{ is $\mu_H$-semistable with } \mu_H(E) \leq \mu
\rangle.
\end{align*}
Let $\alpha>0$ be a positive real number. 
For a coherent sheaf $E$ on $X$, the inequality
$ \mu_{\alpha H}(E)>B_0 \cdot \alpha H $ holds
if and only if   
$\mu_{H}(E)>\mu$ holds. 
Therefore, we have $\mcA_\mu =\mcA(B_0,\alpha H)$ for all $\alpha>0$. 
Then we consider a pair 
\[ \sigma_\alpha=(Z_\alpha, \mcA_\mu)\]
of a group homomorphism  
$Z_\alpha \coloneqq Z_{B_0, \alpha H}:\Halg(X,\bZ) \to \bC$
and the heart $\mcA_\mu$ of a bounded t-structure of $D^b(X)$.

\begin{lem}\label{lem:computation of central charge}
For an element $v=(r,D,s) \in \Halg(X,\bZ)$, we have
\[ Z_\alpha(v)=B_0D-s-\frac{1}{2}rB_0^2+rd \alpha^2+\alpha\sqrt{-1}(D-rB_0)H.   \]
If $r \neq 0$ holds, we have 
\[ Z_{\alpha}(v)=\frac{1}{2r}\left(\delta^2+2\alpha^2 r^2 d-(D-rB_0)^2 \right)+\alpha r \sqrt{-1}(\mu_H(\delta)-\mu). \]
\end{lem}
\begin{proof}
By (\ref{eq:slope of delta}), 
the claim is deduced from the direct computation as in \cite[Section 6]{bri08}.
\end{proof}

We define a real number $\alpha_0$ as
\[ \alpha_0 \coloneqq \frac{1}{r_0\sqrt{d}}>0. \]

When $\alpha=\alpha_0$, we have the following statement by Lemma \ref{lem:computation of central charge}.

\begin{cor}\label{cor:discrete}
We have
\[ Z_{\alpha_0}(\Halg(X,\bZ)) \subset \frac{1}{r^2_0}\left(\bZ \oplus \frac{1}{\sqrt{d}}\bZ\right). \]
\end{cor}

Following  the arguments of the proof of {\cite[Lemma 6.2]{bri08}}, we have the following remark:

\begin{rmk}\label{rem:estimation of central charge}
Assume that $\alpha \geq \alpha_0$ holds.
Then the following statements hold.
\begin{itemize}
\item[(1)]If $E$ is a torsion sheaf on $X$, we have $Z_\alpha(E) \in \bH \cup \bR_{<0}$. 
\item[(2)]If $E$ is a $\mu_H$-semistable torsion free sheaf on $X$ with $\mu_H(E)>\mu$ (resp. $\mu_H(E)<\mu$), we have 
$Z_\alpha(E) \in \bH$ (resp. $Z_\alpha(E[1]) \in \bH$).
\end{itemize}
Let $E$ be a $\mu_H$-stable torsion free sheaf on $X$
with $\mu_H(E)=\mu$.  Note that $v(E)^2 \geq -2$ and $(D-rB_0)H=0$ hold, where we put $v(E)=(r,D,s)$. By the Hodge index theorem, we have 
$(D-rB_0)^2 \geq 0$.  Then $(D-rB_0)^2=0$ holds if and only if $D=rB_0$ holds. Since $v(E) \in \Delta^+_\mu(X)$, we have $r \geq r_0$. 
By Lemma \ref{lem:computation of central charge}, we obtain
\begin{align*}
Z_\alpha(E)&=\frac{1}{2r}\left(v(E)^2+2\alpha^2 r^2 d-(D-rB_0)^2 \right)\\
&\geq \frac{1}{2r}\left(-2+2\alpha^2 r^2_0d  \right)\geq 0.
\end{align*} 
Therefore, $Z_\alpha(E)=0$ holds if and only if $\alpha=\alpha_0$ and $v(E)=\delta_0$ hold. Note that if $v(E)=\delta_0$ holds, we have $E \simeq A$.
\end{rmk}

By Remark \ref{rem:estimation of central charge}, we have the following proposition:

\begin{prop}\label{prop:a>a_0}
The following statements hold\textup{:}
\begin{itemize}
\item[(1)]For $\alpha > \alpha_0$, we have $\sigma_\alpha \in \Stab^*(X)$. 
\item[(2)] For any non-zero object $E \in \mcA_\mu$, we have $Z_{\alpha_0}(E) \in  \bH \cup \bR_{\leq 0}$. 
\item[(3)] For $E \in \mcA_\mu$, the equality $Z_{\alpha_0}(E)=0$ holds if and only if $E \simeq A[1]^{\oplus N}$ holds for some integer $N \geq 0$.  
\end{itemize}
\end{prop}

Note that the first statement (1) is deduced from Theorem \ref{thm:non-existence of certain spherical} and Remark \ref{rem:estimation of central charge}. The second statement (2) follows from Remark \ref{rem:estimation of central charge}. The third statement (3) also follows from Remark \ref{rem:estimation of central charge} and the fact that $A$ is spherical. 
By Remark \ref{rem:slicing and heart}, for $\alpha>\alpha_0$, we obtain the slicing $\mcP_\alpha$ of $D^b(X)$ constructed from $(Z_\alpha, \mcA_\mu)$.
On the other hand, we have the following remark:

\begin{rmk}
By Proposition \ref{prop:a>a_0} (3), $\sigma_{\alpha_0}$ does not give a pre-stability condition on $D^b(X)$. 
We may regard the data $\sigma_{\alpha_0}$ as a degeneration of stability conditions $\{\sigma_\alpha\}_{\alpha>\alpha_0}$. We will make this point of view rigorous by using the theory of lax stability conditions.
\end{rmk}

We also construct the slicing $\mcP_{\alpha_0}$ of $D^b(X)$ from the data $\sigma_{\alpha_0}$.
Putting $\arg(0)\coloneqq \pi$, we consider the map $\arg \colon \bH \cup \bR_{\leq 0} \to (0, \pi]$. 
For simplicity, we write 
\[ \phi_{\alpha_0}(E) \coloneqq \frac{1}{\pi}\arg Z_{\alpha_0}(E) \in (0,1]  \]
for a non-zero object $E \in \mcA_\mu$.
A non-zero object $E$ of $\mcA_\mu$ is $Z_{\alpha_0}$-\textit{semistable} 
if we have 
\begin{equation} \label{eq:def-Zstab}
    \phi_{\alpha_0}(F) \leq \phi_{\alpha_0}(G) 
\end{equation}
for any short exact sequence 
\[
0 \to F \to E \to G \to 0
\]
in $\mcA_\mu$. We say that $E$ is $Z_{\alpha_0}$-\textit{stable} if the inequality (\ref{eq:def-Zstab}) is always strict.  

For a real number $\phi \in (0,1]$, we define
\[ \mcP_{\alpha_0}(\phi) \coloneqq \{ E \in \mcA_\mu : \text{$E$ is $Z_{\alpha_0}$-semistable with $\phi_{\alpha_0}(E)=\phi$}\} \cup \{0\}.  \]
A general case is defined by the property $\mcP_{\alpha_0}(\phi+1)=\mcP_{\alpha_0}(\phi)[1]$ for $\phi \in \bR$.
Then, by the abuse of notation, we also denote $\sigma_{\alpha_0}=(Z_{\alpha_0}, \mcP_{\alpha_0})$.

\begin{rmk}\label{rem:P(1) is abelian}
By Remark \ref{rem:slicing and heart} and Lemma \ref{lem:computation of central charge}, we have 
$\mcP_\alpha(1)=\mcP_{\alpha_0}(1)$ for any $\alpha > \alpha_0$. In particular, $\mcP_{\alpha_0}(1)$ is an abelian category.
\end{rmk}

We also have the weak version of the see-saw property.

\begin{rmk}
For an exact sequence 
\[ 0 \to F \to E \to G \to 0\]
in the abelian category $\mcA_\mu$, the inequality 
\[ \phi_{\alpha_0}(F) \geq \phi_{\alpha_0}(E)   ~\text{(resp. $\phi_{\alpha_0}(F) \leq \phi_{\alpha_0}(E)$)} \]
holds if and only if the inequality
\[ \phi_{\alpha_0}(E) \geq \phi_{\alpha_0}(G)   ~\text{(resp. $\phi_{\alpha_0}(E) \leq \phi_{\alpha_0}(G)$)} \]
holds.
This property is called the \textit{weak see-saw property}.
\end{rmk}

As in the case of the usual slope stability, we can show the following lemma using the weak see-saw property:

\begin{lem}\label{lem:see-saw}
The following statements hold\textup{:} 
\begin{enumerate}
\item Let $E$ and $F$ be $Z_{\alpha_0}$-semistable objects in $\mcA_\mu$ with 
 \[\phi_{\alpha_0}(E) > \phi_{\alpha_0}(F).\]
Then we have $\Hom(E, F)=0$. 
\item If $E$ is a $Z_{\alpha_0}$-stable object, then we have $\Hom(E, E)=\bC$. 
\end{enumerate}
\end{lem}

The Harder--Narasimhan property for $\sigma_{\alpha_0}=(Z_{\alpha_0}, \mcA_\mu)$ also follows from a standard argument: 
\begin{lem} \label{lem:HN}
 For any object $E \in \mcA_\mu$, there is a filtration 
    \[ 0=E_0 \subset E_1 \subset \cdots \subset E_n=E \]
    in $\mcA_\mu$ such that for any $1 \leq i \leq n$, the quotient object $A_i\coloneqq E_i/E_{i-1}$ is $Z_{\alpha_0}$-semistable  with 
    \[\phi_{\alpha_0}(A_1)> \cdots > \phi_{\alpha_0}(A_n).\]
    In particular, the collection $\mcP_{\alpha_0}=\{\mcP_{\alpha_0}(\phi)\}_{\phi \in \bR}$ is a slicing of $D^b(X)$.

\end{lem}
\begin{proof}
The first statement follows from \cite[Proposition 2.8]{liu22}. 
Indeed, we know that 
$\mcA_\mu$ is noetherian by \cite[Proposition 5.0.1]{AP}, Lemma \ref{lem:computation of central charge} and Proposition \ref{prop:a>a_0}. 
The second statement follows from the first statement and Lemma \ref{lem:see-saw}. 
\end{proof}

If the slicing $\mcP_{\alpha_0}$ is locally finite, then $\sigma_{\alpha_0}$ is a lax pre-stability condition on $D^b(X)$. 
We also use the same terminology as in the text below Definition \ref{def:pre-stability} as well as in Definition \ref{def:mass} and Definition \ref{def:massless} for $\sigma_\alpha=(Z_{\alpha_0}, \mcP_{\alpha_0})$. 
To prove the local finiteness, we study properties of massless objects.

\begin{lem}\label{lem:central charge is zero}
Let $I$ be an interval satisfying $I \subset (a,a+1]$ for some $a \in \bR$. 
For a non-zero object $E \in \mcP_{\alpha_0}(I)$, $E$ is massless if and only if $Z_{\alpha_0}(E)=0$ holds. 
\end{lem}
\begin{proof}
If $E$ is massless, the equality $Z_{\alpha_0}(E)=0$ follows from the inequality \[ |Z_{\alpha_0}(E)| \leq m_{\sigma_{\alpha_0}}(E)=0.\] 
Assume that $Z_{\alpha_0}(E)=0$ holds.
Take the Harder--Narasimhan filtration 
\[0=E_0 \to E_1 \to E_2 \to \cdots \to E_{n-1} \to E_n=E \]
of $E$ with the $\sigma_{\alpha_0}$-semistable factors $A_i \in \mcP_{\alpha_0}(\phi_i)$ and $\phi_1> \cdots > \phi_n$ in the interval $I$.
By the assumption $Z_{\alpha_0}(E)=0$, we have
\[ Z_{\alpha_0}(A_1)=-\sum_{i=2}^nZ_{\alpha_0}(A_i). \]
Since $I \subset (a,a+1]$, we obtain
\[ Z_{\alpha_0}(A_1)=0,~\sum_{i=2}^nZ_{\alpha_0}(A_i)=0.\]
Inductively, we get
\[ Z_{\alpha_0}(A_1)=Z_{\alpha_0}(A_2)=\cdots=Z_{\alpha_0}(A_n)=0.  \]
Therefore, we have $m_{\sigma_{\alpha_0}}(E)=0$.
\end{proof}

\begin{lem}\label{lem:massless is semistable}
Let $I$ be an interval satisfying $I \subset (a,a+1]$ for some $a \in \bR$. 
If $E \in \mcP_{\alpha_0}(I)$ is massless, then $E$ is $\sigma_{\alpha_0}$-semistable and $E \simeq A[\ell]^{\oplus N}$ holds for some integer $\ell$ and a non-negative integer $N$. 
\end{lem}
\begin{proof}
Assume that $m_{\sigma_{\alpha_0}}(E)=0$ holds.
Take the Harder--Narasimhan filtration 
\[0=E_0 \to E_1 \to E_2 \to \cdots \to E_{n-1} \to E_n=E \]
of $E$ with the $\sigma_{\alpha_0}$-semistable factors $A_i \in \mcP_{\alpha_0}(\phi_i)$ and $\phi_1> \cdots > \phi_n$ in the interval $I$.
Since $m_{\sigma_{\alpha_0}}(E)=0$, we have
\[ Z_{\alpha_0}(A_1)=Z_{\alpha_0}(A_2)=\cdots=Z_{\alpha_0}(A_n)=0.  \]
Note that $\phi_{\sigma_{\alpha_0}}(A_i)$ is an integer for any $1 \leq i \leq n$.
Since $I \subset (a,a+1]$, we have $n=1$.
Therefore, $E$ is $\sigma_{\alpha_0}$-semistable.  
By Proposition \ref{prop:a>a_0} (3), the second assertion also holds.
\end{proof}

\begin{lem}\label{lem:sub of massless}
Let $I$ be an interval satisfying $I \subset (a,a+1]$ for some $a \in \bR$. 
For a massless object $E \in \mcP_{\alpha_0}(I)$ and a strict short exact sequence
\[ 0 \to F \to E \to G \to 0  \]
in the quasi-abelian category $\mcP_{\alpha_0}(I)$, the objects
$F$ and $G$ are also massless.
\end{lem}
\begin{proof}
By the assumption $m_{\sigma_{\alpha_0}}(E)=0$, we have $Z_{\alpha_0}(E)=0$ by Lemma \ref{lem:central charge is zero}.
Therefore, $Z_{\alpha_0}(F)=-Z_{\alpha_0}(G)$ holds. 
Since $I \subset (a,a+1]$, we obtain
\[ Z_{\alpha_0}(F)=Z_{\alpha_0}(G)=0. \]
By Lemma \ref{lem:central charge is zero}, $F$ and $G$ are also massless.
\end{proof}

\begin{rmk}\label{rem:A is sigma-stable}
By Proposition \ref{prop:a>a_0} (3), Lemma \ref{lem:central charge is zero} and Lemma \ref{lem:sub of massless}, $A[1]$ is a simple object in $\mcP_{\alpha_0}(1)$. In particular, $A[1]$ is a $\sigma_{\alpha_0}$-stable object of phase one.
\end{rmk}

By Proposition \ref{prop:a>a_0} (3), Lemma \ref{lem:central charge is zero}, Lemma \ref{lem:massless is semistable} and Lemma \ref{lem:sub of massless}, we obtain the following lemma:

\begin{lem}\label{lem:chain in mass less}
Let $I$ be an interval satisfying $I \subset (a, a+1]$ for some $a \in \bR$.
For a massless object $E \in \mcP_{\alpha_0}(I)$, a chain 
\[ \cdots \subset E_2 \subset E_1 \subset E_0=E \]
of strict subobjects of $E$ in the quasi-abelian category $\mcP_{\alpha_0}(I)$ terminates.
\end{lem}

We are now ready to prove the local finiteness of the slicing $\mcP_{\alpha_0}$:

\begin{prop} \label{prop:locfin}
The slicing $\mcP_{\alpha_0}$ is locally finite. 
In particular, $\sigma_{\alpha_0}$ is a lax pre-stability condition on $D^b(X)$.
\end{prop}
\begin{proof}
Take a real number $0< \epsilon < 1/2$.
We fix a real number $\phi \in \bR$.
We consider a quasi-abelian category $\mcP_{\alpha_0}(\phi-\epsilon,\phi+\epsilon)$.
Take a chain 
\[ \cdots \subset E_2 \subset E_1 \subset E_0=E \]
of strict subobjects of $E$ in $\mcP_{\alpha_0}(\phi-\epsilon,\phi+\epsilon)$.
If there exists $N \geq 0$ such that $E_N$ is a massless object, the above chain terminates by Lemma \ref{lem:chain in mass less}. 
Assume that $E_i$ is massive for all $i \geq 0$.
By the argument same as the proof of \cite[Lemma 4.4]{bri08}, the above chain terminates.
\end{proof}

To prove the support property, we will use the following lemmas: 

\begin{lem}[{\cite[Lemma 8.1]{bri08}}] \label{lem:CS}
There exists a constant $C>0$ such that the following property holds. 
Let $v \in \Halg(X, \bZ)_\bR$ be a vector satisfying one of the following conditions\textup{:} 
\begin{itemize}
    \item $v^2 \geq 0$, 
    \item $v^2=-2$ and $Z_{\alpha_0}(v) \neq 0$. 
\end{itemize}
Then we have 
\begin{equation} \label{eq:support}
\|v \| \leq C|Z_{\alpha_0}(v)|. 
\end{equation}
\end{lem}
\begin{proof}
This is essentially proved in \cite[Lemma 8.1]{bri08}. 
Note that 
\[
Z_{\alpha_0}(-)=\langle\Omega, - \rangle 
\]
holds, where we put $\Omega \coloneqq \exp(B_0+\alpha_0 \sqrt{-1} H)$.
By the first part of the proof of \textit{loc. cit.}, there exists a constant $C>0$ such that the inequality (\ref{eq:support}) holds for any $v \in \Halg(X, \bZ)_\bR$ with $v^2 \geq 0$. 
Furthermore, the second part of the proof of \textit{loc. cit.} also shows that the same inequality holds for a vector $v$ with $v^2=-2$ as long as it satisfies $\langle \Omega, v \rangle \neq 0$. 
\end{proof}

\begin{lem}\label{lem:JH}
Let $E \in \mcA_\mu$ be a $\sigma_{\alpha_0}$-stable object. 
Then there exists a filtration 
\[ 0=E_0 \subset E_1 \subset \cdots \subset E_n=E \]
in $\mcA_\mu$ such that for any $1 \leq i \leq n$, the quotient object $A_i\coloneqq E_i/E_{i-1}$ is $Z_{\alpha_0}$-stable  with $\phi_{\alpha_0}(A_i)=\phi_{\alpha_0}(E)$. 
\end{lem}
\begin{proof}
First, assume that $\phi_{\alpha_0}(E)=1$, i.e., $\Im Z_{\alpha_0}(E)=0$. Consider an exact sequence 
\[
0 \to F \to E \to G \to 0
\]
in $\mcA_\mu$ such that objects $F$ and $G$ are non-zero objects. Then we have $\phi_{\alpha_0}(F)=\phi_{\alpha_0}(G)=1$, i.e., it defines an exact sequence in $\mcP_{\alpha_0}(1)$. As we assume $E \in \mcP_{\alpha_0}(1)$ is $\sigma_{\alpha_0}$-stable, there is no such exact sequence. This proves that $E$ is already $Z_{\alpha_0}$-stable. 

Let us now assume that $\Im Z_{\alpha_0}(E)\neq 0$.
We prove the assertion by induction on $\Im Z_{\alpha_0}(E)$ 
(recall that  the image of $\Im Z_{\alpha_0}$ is discrete). 
We may assume that $E$ is strictly $Z_{\alpha_0}$-semistable. Then there exists an exact sequence 
\[
0 \to F \to E \to G \to 0
\]
such that $F \neq 0, G \neq 0$, and 
$\phi_{\alpha_0}(F)=\phi_{\alpha_0}(G)$. In particular, we have 
\[
0 < \Im Z_{\alpha_0}(F)<\Im Z_{\alpha_0}(E),~ 0<\Im Z_{\alpha_0}(G) < \Im Z_{\alpha_0}(E). 
\]
By the induction hypothesis, both $F$ and $G$ have the required property, hence so does $E$. 
\end{proof}

Now we are ready to prove the support property: 
\begin{prop} \label{prop:support}
The lax pre-stability condition $\sigma_{\alpha_0}$ satisfies a support property, i.e., 
there exists a constant $C>0$ such that for any massive $\sigma_{\alpha_0}$-stable object, we have 
\[
\|v(E)\| \leq C|Z_{\alpha_0}(E)|. 
\]
\end{prop}
\begin{proof}
By taking shifts, we may assume that $E \in \mcA_\mu$. 
By Lemma \ref{lem:JH}, we may assume that $E$ is $Z_{\alpha_0}$-stable and hence we have $\Hom(E, E)=\bC$ (see Lemma \ref{lem:see-saw}). 
By the Riemann--Roch formula and Serre duality, we have 
\[
-v^2=\chi(E, E) \leq \hom(E, E)+\ext^2(E, E)=2. 
\]
Furthermore, since we assume that $E$ is massive, we have $Z_{\alpha_0}(v) \neq 0$. 
Now, the assertion holds by Lemma \ref{lem:CS}. 
\end{proof}

\begin{rmk}
As we will see in Proposition \ref{prop:delta-lax support property}, the lax stability condition $\sigma_{\alpha_0}$ satisfies a stronger version of the support property. 
\end{rmk}

\begin{lem} \label{lem:limit}
In the generalized metric space $\Stab^l(D^b(X))$, we have
\[
\lim_{\alpha\to\alpha_0+0}\sigma_\alpha=\sigma_{\alpha_0}.
\]
\end{lem}
\begin{proof}
By (\ref{eq:metric}), it is enough to show that 
\[ \lim_{\alpha \to \alpha_0+0}\|Z_\alpha-Z_{\alpha_0} \|=0,~\lim_{\alpha \to \alpha_0+0}d(\mcP_\alpha, \mcP_{\alpha_0})=0. \]
In fact, the first equality is deduced from (\ref{eq:operator norm}) and Lemma \ref{lem:computation of central charge}.

Next, we prove the second equality.
By {\cite[Lemma 6.1]{bri}}, it is sufficient to show that 
for any $\epsilon>0$, 
there exists $\delta>0$ such that  $0<\alpha-\alpha_0<\delta$ implies
\[
\mcP_\alpha(\phi)\subset\mcP_{\alpha_0}([\phi-\epsilon, \phi+\epsilon])
\text{ for any $\phi\in(0,1]$}. 
\]
By Proposition \ref{prop:support} and \cite[Proposition 4.11]{bppw}, 
we also have 
\begin{equation}\label{eq:semi-norm}
\lim_{\alpha\to\alpha_0+0}||Z_\alpha-Z_{\alpha_0}||_{\sigma_{\alpha_0}}=0,
\end{equation}
where we define
\[
||W||_{\sigma_{\alpha_0}}
\coloneqq \sup
\left\{
\frac{|W(E)|}{|Z_{\alpha_0}(E)|}
 : 
\text{massive stable }E\in\mcP_{\alpha_0}(\phi),~\phi\in\bR
\right\}
\]
for a group homomorphism $W \colon \Halg(X, \bZ) \to \bC$. 

For simplicity, we write 
\[\phi_\alpha(E)= \frac{1}{\pi}\arg Z_\alpha(E)\]
for a non-zero object $E\in\mcA_\mu$. 
Note that $\phi_\alpha(E)=\phi_{\alpha_0}(E)=1$ holds for any massless $\sigma_{\alpha_0}$-stable object $E$ in $\mcA_\mu$ by Lemma \ref{lem:computation of central charge}.

Take any $\epsilon>0$. By (\ref{eq:semi-norm}), there exists $\delta>0$ such that $0< \alpha-\alpha_0<\delta$ implies $\|Z_\alpha-Z_{\alpha_0} \|_{\sigma_{\alpha_0}}<\sin(\pi \epsilon)$. 
As in \cite[Remark 4.50]{Huybrechts_2014}, we obtain 
\[ |\phi_\alpha(E)-\phi_{\alpha_0}(E)|<\epsilon \]
for any $\sigma_{\alpha_0}$-stable object $E \in \mcA_\mu$.

Take $\phi \in (0,1]$ and a non-zero object $F \in \mcP_\alpha(\phi)$. Consider the 
Harder--Narasimhan filtration of $F$ with respect to $\sigma_{\alpha_0}$:
\[0=F_0 \subset F_1 \subset F_2 \subset \cdots \subset F_{n-1} \subset F_n=F \]
with the $\sigma_{\alpha_0}$-semistable factor $A_i\coloneqq F_i/F_{i-1}$ of phase $\phi_i$ for $1 \leq i \leq n$. By Proposition \ref{prop:locfin}, there is a strict filtration  
\[ 0=C_0 \subset C_1 \subset C_2 \subset \cdots \subset C_{k-1} \subset C_k=A_n \]
of $A_n$ in $\mcP_{\alpha_0}(\phi_n)$ such that the strict quotient $C_i/C_{i-1}$ is $\sigma_{\alpha_0}$-stable for $1 \leq i \leq k$. For each $i$, we have
\[ |\phi_\alpha(C_i/C_{i-1})-\phi_{\alpha_0}(A_n)|=|\phi_\alpha(C_i/C_{i-1})-\phi_{\alpha_0}(C_i/C_{i-1})|<\epsilon. \]
Since $Z_\alpha(A_n)=\sum_{i=1}^{k}Z_\alpha(C_i/C_{i-1})$ holds, we obtain 
\[  |\phi_\alpha(A_n) - \phi_{\alpha_0}(A_n)|<\epsilon. \]
By the definition of $A_n$, we have a surjection $F \twoheadrightarrow A_n$ in $\mcA_\mu$. Since $F$ is $\sigma_\alpha$-semistable, we have
\[ \phi=\phi_\alpha(F) \leq \phi_\alpha(A_n), \]
hence $\phi-\epsilon<\phi_{\alpha_0}(A_n)$.
Similar computations imply $\phi_{\alpha_0}(A_1)<\phi+\epsilon$, 
thus we have $F\in\mcP_{\alpha_0}((\phi-\epsilon, \phi+\epsilon))$, 
which completes the proof. 
\end{proof}

Let $\Stab^L(X)$ denote the connected component of the space $\Stab^L(D^b(X))$ containing $\Stab^*(X)$. 

\begin{thm}\label{thm:laxstab}
We have $\sigma_{\alpha_0}=(Z_{\alpha_0}, \mcP_{\alpha_0}) \in \Stab^L(X)\setminus \Stab^*(X)$.
In particular, $\bP m (\sigma_{\alpha_0})$ lies on the boundary of the closure of $\bP m (\Stab^*(X)/\bC)$. 
\end{thm}
\begin{proof}
The first assertion follows from Proposition \ref{prop:a>a_0}, Lemma \ref{lem:see-saw}, Lemma \ref{lem:HN}, Proposition \ref{prop:locfin}, Proposition \ref{prop:support}, and Lemma \ref{lem:limit}. 
The second assertion then follows from the continuity of the map $\bP m$.
\end{proof}

Note that the sequence yields a normal deformation of $\sigma_{\alpha_0}$ in the sense of \cite[\S 6.1]{bppw}. 

\begin{rmk}
The massless subcategory of $\sigma_{\alpha_0}$ is precisely $\langle A\rangle$ of rank one in the sense of \cite{bppw}. 
\end{rmk}

\subsection{Comparing points on the boundary}
We keep the same setting as in the previous subsection.
In this subsection, we prove the following proposition: 

\begin{prop}\label{prop:laxvshom}
Let $A$ be a spherical object on $X$. Assume that $A$ is $\sigma$-semistable for some stability condition $\sigma$ on $D^b(X)$. 
Then we have
\[ \bP m([\sigma_{\alpha_0}]) \neq h_A \]
in the infinite dimensional projective space $\bP^{\mcS}_{\geq 0}$.
\end{prop}

We first prepare two lemmas: 
\begin{lem}[{\cite[Lemma 6.24]{ms17}}] \label{lem:finitewall}
For any spherical class $\delta \in \Delta(X)$, 
there are only finitely many walls with respect to $\delta$ which intersect with the line $\{\sigma_\alpha\}_{\alpha > \alpha_0}$. 
\end{lem}

\begin{lem} \label{lem:laxnonempty}
For any spherical class $\delta \in \Delta(X)$, there exists a $\sigma_{\alpha_0}$-semistable spherical object $E$ with $v(E)=\delta$. 
\end{lem}
\begin{proof}
By Lemma \ref{lem:finitewall}, there exists a real number $\alpha_1>\alpha_0$ such that there are no walls intersecting the set $\{\sigma_{\alpha}\}_{\alpha \in [\alpha_0, \alpha_1)}$. 
Pick any $\alpha \in (\alpha_0, \alpha_1)$. By Lemma \ref{lem:sphexist}, there exists a $\sigma_\alpha$-semistable object $E$ with $v(E)=\delta$. It follows that $E$ remains semistable with respect to $\sigma_{\alpha_0}$. 
\end{proof}

For $\ell \in \bZ$, we consider a spherical class
\begin{align*}
\delta_{\ell}&\coloneqq e^{\ell H} \cdot \delta_0\\
&=\delta_0+\ell(0,r_0H,\ell HD_0+d\ell r_0)
\end{align*}
on $X$.
By Lemma \ref{lem:laxnonempty}, there is a $\sigma_{\alpha_0}$-semistable object $E_{\ell}$ with $v(E_{\ell})=\delta_{\ell}$ for any integer $\ell \in \bZ$.
Since $Z_{\alpha_0}(\delta_0)=0$ holds, we obtain
\[ 
Z_{\alpha_0}(E_\ell)
=-r_0d\ell^2+2\sqrt{d}\ell \sqrt{-1} 
\]
by the direct computation. Then we have
\[ m_{\sigma_{\alpha_0}}(E_\ell)=|\ell|\sqrt{d} \cdot \sqrt{r^2_0d\ell^2+4}. \]

We will also use the following lemma in the proof of Proposition \ref{prop:laxvshom}.

\begin{lem}\label{lem:prime number}
Let $a$ be a positive integer. For an integer, we define 
\[ f(x)\coloneqq ax^2+4. \]
Then there are a prime number $p$ and positive integers $x_0$ and $x_1$ satisfying the following conditions.
\begin{itemize}
\item[(1)]$p \nmid f(x_0)$ holds.
\item[(2)] $p \mid f(x_1)$ and $p^2 \nmid f(x_1)$ hold.
\end{itemize}
\end{lem}
\begin{proof}
Since $4=2^2$ is a square number, for an odd prime number $p$ greater than $a$, there exists a non-zero integer $x$ satisfying
$p \mid f(x)$ holds if and only if 
\[ \left(\frac{a}{p} \right)=\left(\frac{-1}{p} \right)=1 \]
holds, where we use the Legendre symbol for quadratic residues. For an integer $k$, we regard 
\[ \left(\frac{kp}{p} \right)=0. \]
By the first supplement to the quadratic reciprocity law, the condition 
\[ \left(\frac{-1}{p} \right)=1\]
holds if and only if $p \equiv 1 \pmod {4}$ holds.

The positive integer $a$ can be expressed in a unique way 
as $a=m^2 b$, where $m$ is a positive integer and $b$ is a square free number. Then we have
\[ \left(\frac{a}{p}\right)=\left(\frac{m^2}{p} \right)\left(\frac{b}{p} \right)=\left(\frac{b}{p} \right).\]
Consider the prime factorization $b=q_1 \cdots q_n$ of $b$
satisfying the inequality $q_1<\cdots<q_n$. By the law of quadratic reciprocity, we obtain
\[ \left(\frac{b}{p} \right)=\prod_{i=1}^{n} \left(\frac{q_i}{p} \right)
    =\prod_{i=1}^{n}\left(\frac{p}{q_i} \right)(-1)^{\frac{p-1}{2}\cdot\frac{q_i-1}{2}}.\]
If $p$ satisfies $p \equiv 1 \pmod{4}$, we have $(-1)^{\frac{p-1}{2}\cdot\frac{q_i-1}{2}}=1$ for $1 \leq i \leq n$. For $1 \leq i \leq n$, we take a positive integer $r_i$ such that 
\[ \left(\frac{r_i}{q_i} \right)=1 \]
holds. 
Consider the congruence equation
\begin{equation}\label{eq:congruence equation}
\left\{ \,
    \begin{aligned}
     x &\equiv r_1 \pmod{q_1} \\
    &\vdots \\
     x &\equiv r_n \pmod{q_n}\\
     x &\equiv 1 \pmod{4}. 
    \end{aligned}
\right.
\end{equation}
Note that if $q_1=2$ holds, we have $r_1=1$.
By the Chinese remainder theorem, there exists a unique positive integer $r$ modulo $q$ such that $r$ satisfies the congruence equation (\ref{eq:congruence equation}),  where $q$ is a positive integer defined by
\[q=\left\{
\begin{array}{ll}
4b & (q_1 \neq 2)\\
2b & (q_1=2)
\end{array}
\right.\]
Since $r$ satisfies (\ref{eq:congruence equation}), we have $\gcd(q,r)=1$. 
By Dirichlet's theorem on arithmetic progressions, there are infinitely many prime numbers satisfying the congruence equation (\ref{eq:congruence equation}).

Take a positive integer $x_0$.  By the above argument, we can take a prime number $p$ satisfying $p>\max\{a,4, f(x_0)\}$ and (\ref{eq:congruence equation}).
Then the condition (1) is satisfied. Since $p$ satisfies (\ref{eq:congruence equation}), we have 
\[ \left(\frac{a}{p} \right)=\left(\frac{-1}{p} \right)=1.\] 
Therefore, there exists a positive integer $x_1$ satisfying $p \mid f(x_1)$. 
Since $p$ is greater than $a$ and $4$, we have $p \nmid x_1$. We compute $f(x_1+p)$ as
\begin{align*}
    f(x_1+p)&=a(x_1+p)^2+4\\
    &=f(x_1)+2apx_1+ap^2\\
    &\equiv f(x_1)+2apx_1 \pmod {p^2}.
\end{align*}
If $p^2 \mid f(x_1)$ and $p^2 \mid f(x_1+p)$ hold, then we obtain $p \mid 2a x_1$ and this is a contradiction. 
Since $f(x_1)\equiv f(x_1+p) \pmod{p}$, we can choose a positive integer $x_1$ satisfying the condition (2). 
\end{proof}

\vspace{3mm}
\begin{proof}[Proof of Proposition \ref{prop:laxvshom}]
We put $f(x)\coloneqq r_0^2dx^2+4$.
For an integer $\ell \neq 0$, we have
\[ m_{\sigma_{\alpha_0}}(E_\ell)=|\ell|\sqrt{d} \cdot \sqrt{f(\ell)}>0. \]
By Lemma \ref{lem:prime number}, there are a prime number $p$ and positive integers $\ell_0$ and $\ell_1$ satisfying the following conditions:
\begin{itemize}
\item $p \nmid f(\ell_0)$  holds.
\item  $p \mid f(\ell_1)$ and $p^2 \nmid f(\ell_1)$ hold.
\end{itemize}
Therefore, the real number 
\[ \frac{m_{\sigma_{\alpha_0}}(E_{\ell_1})}{m_{\sigma_{\alpha_0}}(E_{\ell_0})}=\frac{|\ell_1|}{|\ell_0|}\cdot \sqrt{\frac{f(\ell_1)}{f(\ell_0)}}\]
is irrational. 

Let $A$ be a spherical object on $X$. 
Assume that \[\bP m(\sigma_{\alpha_0}) = h_A\] holds.
Then we obtain 
\[ \frac{m_{\sigma_{\alpha_0}}(E_{\ell_1})}{m_{\sigma_{\alpha_0}}(E_{\ell_0})}=\frac{h_A(E_{\ell_1})}{h_A(E_{\ell_0})} \in \bQ. \]
This contradicts to the argument in the first paragraph.
\end{proof}

\subsection{$\epsilon$-lax support property}\label{sec:delta-lax support}
In order to prove a deformation result for lax stability conditions, the following variant of the usual support property (see Definition \ref{def:lax support}) is introduced in \cite{bppw}: 

\begin{dfn}\label{def:delta-lax support}
Take a real number $\epsilon>0$.
A lax pre-stability condition $\sigma=(Z,\mcP)$ on $\mcD$ satisfies the $\epsilon$-\textit{lax support property} if there is a constant $C>0$ such that for any real number $\phi \in \bR$ and any massive indecomposable object $E \in \mcP(\phi-\epsilon, \phi+\epsilon)$, we have
\[ \| \cl(E) \| < C|Z(E)|. \]
\end{dfn}

In this subsection, we prove the following proposition:

\begin{prop}\label{prop:delta-lax support property}
Take a real number $0<\epsilon<1/4$.
Then the lax stability condition $\sigma_{\alpha_0}$ satisfies the $\epsilon$-lax support property.
\end{prop}

We use the following norm on the $\bR$-vector space $\Halg(X,\bZ)_{\bR}$.

\begin{dfn}
Fix a $\bR$-basis $\{v_1, \cdots v_{\rho(X)+2}\}$ of $\Halg(X,\bZ)_\bR$ satisfying 
\[ v_1=\delta_0,~\langle v_1, v_i \rangle=0 \]
for $2 \leq i \leq \rho(X)+2$.
For $v \in \Halg(X,\bZ)_\bR$, we define 
\[ \| v \|\coloneqq |\langle v_1,v\rangle | +\sum_{i=2}^{\rho(X)+2}|\langle v_i,v \rangle|. \]
Then $\|-\|$ is a norm on $\Halg(X,\bZ)_\bR$.
\end{dfn}

\begin{rmk}\label{rem:known support inequaility}
By Lemma \ref{lem:massless is semistable}, Proposition \ref{prop:locfin} and Proposition \ref{prop:support}, there is a constant $C>0$ such that the inequality
\[  \|v(E) \|< C|Z_{\alpha_0}(E)| \]
holds for non-zero object $E$ in $D^b(X)$ satisfying one of the following conditions:
\begin{itemize}
\item[(1)] $E \in \mcP_{\alpha_0}(\phi)$ for some $\phi \in \bR \setminus \bZ$,.
\item[(2)]  $E$ is a massive $\sigma_{\alpha_0}$-stable of phase $\phi$ for some $\phi \in \bZ$. 
\end{itemize}
\end{rmk}

Fix $0<\epsilon<1/4$ and take $\phi \in \bR$. 
Let $C>0$ be a constant as in Remark \ref{rem:known support inequaility}.
The following holds:

\begin{lem}\label{lem:lax support without massless}
Let $E \in \mcP_{\alpha_0}(\phi-\epsilon, \phi+\epsilon)$ be a non-zero object. Assume that all $\sigma_{\alpha_0}$-semistable factors of $E$ satisfy the condition $(1)$ or $(2)$ in Remark \ref{rem:known support inequaility}. 
Then we have 
\[  \|v(E) \|< \frac{C}{\cos(2\pi \epsilon)}|Z_{\alpha_0}(E)|. \]
\end{lem}
\begin{proof}
Take the Harder--Narasimhan filtration of $E$ with respect to $\sigma_{\alpha_0}$:
\[0=E_0 \subset E_1 \subset E_2 \subset \cdots \subset E_{n-1} \subset E_n=E \]
with the $\sigma_{\alpha_0}$-semistable factor $A_i\coloneqq E_i/E_{i-1}$ of phase $\phi_i$ for $1 \leq i \leq n$. 
For each $1 \leq i \leq n$, the object $A_i$ satisfies the condition (2) in Remark \ref{rem:known support inequaility} by assumption. Then we obtain
\[ \|v(E)\|\leq \sum_{i=1}^n\|v(A_i)\|<C \sum_{i=1}^n|Z_{\alpha_0}(A_i)|\leq \frac{C}{\cos(2\pi\epsilon)}|Z_{\alpha_0}(E)|.\]
\end{proof}

The quasi-abelian category $\mcP_{\alpha_0}(\phi-\epsilon, \phi+\epsilon)$ is a full subcategory of the heart 
\[ \mcB \coloneqq \mcP_{\alpha_0}((\phi-\epsilon, \phi-\epsilon+1]) \]
of a bounded t-structure on $D^b(X)$. 
We assume that $(\phi-\epsilon, \phi+\epsilon)$ contains an integer. Taking shifts, we may assume that $0 \in (\phi-\epsilon, \phi+\epsilon)$. Now, putting $\arg(0)\coloneqq 0$, we can consider the map 
 $\arg Z_{\alpha_0} \colon \mcB \setminus \{0 \} \to (\phi-\epsilon, \phi-\epsilon+1]$. Similarly to the previous subsections, we write 
 \[ \phi_{\alpha_0}(E) \coloneqq \frac{1}{\pi}\arg Z_{\alpha_0}(E) \]
 for $E \in \mcB \setminus \{0 \}$.
Note that a non-zero object $E$ in $\mcB$ is $\sigma_{\alpha_0}$-semistable if and only if we have $\phi_{\alpha_0}(F) \leq \phi_{\alpha_0}(G)$
for any short exact sequence 
\[
0 \to F \to E \to G \to 0
\]
in $\mcB$.
We define the massless subcategory
\[ \mcB_0 \coloneqq \{ E \in \mcB \colon Z_{\alpha_0}(E)=0 \}  \]
of $\mcB$. By Lemma \ref{lem:massless is semistable}, we have
\[ \mcB_0=\{E \in \mcB \colon E \simeq A^{\oplus N}~\text{for some $N \geq 0$} \}. \]
We define full subcategories 
\begin{align*}
\mcB^\perp_0 &\coloneqq \{ E \in \mcB \colon \Hom(E, \mcB_0)=0 \},\\
 {}^\perp\mcB_0 &\coloneqq \{ E \in \mcB \colon \Hom(\mcB_0, E)=0 \} 
 \end{align*}
of $\mcB$.
Then the following holds:

\begin{lem}
The pairs $(\mcB^\perp_0,\mcB_0)$ and $(\mcB_0,{}^\perp\mcB_0)$ are torsion pairs on $\mcB$.
\end{lem}
\begin{proof}
By Lemma \ref{lem:massless is semistable} and Lemma \ref{lem:sub of massless}, $A$ is simple in $\mcB$. 
First, we prove that $(\mcB^\perp_0,\mcB_0)$ is a torsion pair on $\mcB$. Let $E$ be an object in $\mcB$. We show that the evaluation map $\mathrm{ev} \colon \Hom(A,E) \otimes A \to E$ is injective in $\mcB$. For a basis $\{ f_1, \cdots, f_N\}$ of $\Hom(A,E)$, we can regard the map $\mathrm{ev}\colon \Hom(A,E)\otimes A \to E$ as
\[ \mathrm{ev}=(f_1, \cdots, f_N) \colon A^{\oplus N} \to E. \]
Since $A$ is simple, $f_i$ is injective in $\mcB$ for $1 \leq i \leq N$. For $i, j \in \{1,2, \cdots, N\}$, the intersection $\mathrm{Im}(f_i) \cap \mathrm{Im}(f_j)$ of the images of $f_i$ and $f_j$ is isomorphic to $A$ or $0$. Since $\Hom(A,A)=\bC$, the condition $\mathrm{Im}(f_i) \cap \mathrm{Im}(f_j) \simeq A$
implies that $f_i=\lambda f_j$ holds for some $\lambda \in \bC^*$. The latter condition is equivalent to $i=j$
since $\{ f_1, \cdots, f_N\}$ is a basis of $\Hom(A, E)$. Therefore, the evaluation map $\mathrm{ev} \colon \Hom(A,E) \otimes A \to E$ is injective in $\mcB$. Let $F\coloneqq\Coker (\mathrm{ev})$. 
Since $\Ext^1(A,A)=0$, we have the exact sequence
 \[ 0 \to \Hom(A,E) \otimes \Hom(A,A) \xrightarrow{\mathrm{ev}} \Hom(A,E) \to \Hom(A, F) \to 0, \]
 Since $\mathrm{ev} \colon \Hom(A,E) \otimes \Hom(A,A) \to \Hom(A,E)$ is an isomorphism, we have $\Hom(A,F)=0$. Therefore,we have $F \in \mcB^\perp_0$ and $(\mcB^\perp_0,\mcB_0)$ is a torsion pair on $\mcB$.

Next, we prove that $(\mcB_0,{}^\perp\mcB_0)$ is a torsion pair on $\mcB$. Consider the coevaluation map $\mathrm{coev} \colon E \to \Hom(E,A)^* \otimes A$. Take a basis $\{g_1, \cdots, g_M \}$ of $\Hom(E,A)$. Then we can regard the map $\mathrm{coev} \colon E \to \Hom(E,A)^* \otimes A$ as 
\[ \mathrm{coev}=
\begin{pmatrix}
g_1 \\
\vdots \\
g_M
\end{pmatrix} 
\colon E \to A^{\oplus M}. \]
Note that there is the canonical morphism $\widetilde{\mathrm{coev}} \colon E \to \mathrm{Im}(\mathrm{coev})$ such that $\mathrm{coev}=\iota \circ \widetilde{\mathrm{coev}}$ holds, where $\iota \colon \mathrm{Im}(\mathrm{coev}) \to A^{\oplus M}$ is the natural inclusion.
Since $A$ is simple, there is an integer $0 \leq M' \leq M$ such that 
\[\mathrm{Im}(\mathrm{coev}) \simeq A^{\oplus M'}\]
holds. Then we have $\hom(\mathrm{Im}(\mathrm{coev}),A)=M'$. 
For $1 \leq i \leq M$, consider the morphism $p_i \colon \mathrm{Im}(\mathrm{coev}) \to A$ defined by $p_i \coloneqq \mathrm{pr}_i \circ \iota$, where $\mathrm{pr}_i \colon A^{\oplus M} \to A$ is the $i$-th projection. Since $p_i \circ \widetilde{\mathrm{coev}}=g_i$ holds for $1 \leq i \leq M$, the morphisms $p_1, \cdots, p_M$ are linearly independent. Therefore, we obtain $M'=M$ and $\mathrm{coev} \colon E \to \Hom(E,A)^* \otimes A$ is surjective in $\mcB$. 
Let $F \coloneqq \ker(\mathrm{coev})$ be its kernel. 
Then we obtain an exact sequence 
\[
0 \to \Hom(E, A) \otimes \Hom(A, A) \xrightarrow{\mathrm{coev}} \Hom(E, A) \to \Hom(F, A) \to 0, 
\]
from which we see that $\Hom(F, A)=0$, i.e., $F \in ^{\perp}\mcB_0$. Therefore, $(\mcB_0,{}^\perp\mcB_0)$ is a torsion pair on $\mcB$.
\end{proof}


\begin{rmk}\label{rem:red}
Let $E$ be a non-zero object in $\mcB$. Using the torsion pair $(\mcB^\perp_0,\mcB_0)$, there is a  unique exact sequence
\begin{equation}\label{eq:torsion M}
0 \to F \to E \to A^{\oplus M} \to 0 
\end{equation}
in $\mcB$ such that $\Hom(F, A)=0$.
Similarly, using the torsion pair $(\mcB_0,{}^\perp\mcB_0)$, we have a unique exact sequence
\begin{equation}\label{eq:torsion N}
0 \to A^{\oplus N} \to F \to G \to 0 
\end{equation}
in $\mcB$ such that $\Hom(A, G)=0$.
We put  $E_{\mathrm{red}} \coloneqq G$.
Note that 
\[ \Hom(A,E_\mathrm{red})=\Hom(E_\mathrm{red},A)=0\]
holds.
By Proposition \ref{prop:a>a_0} (3), we have
\[ Z_{\alpha_0}(E)=Z_{\alpha_0}(E_\mathrm{red}). \]
\end{rmk}

The following key lemma compares the norm of $E$ and that of $E_{\mathrm{red}}$ for an indecomposable $E \in \mcB$: 
\begin{lem}\label{lem:norm and red}
Let $E$ be a non-zero object in $\mcB$. Assume that $A$ is not a direct summand of $E$. Then we have
\[ |\langle v(A), v(E) \rangle| \leq  \langle v(A), v(E_\mathrm{red}) \rangle. \]
\end{lem}
\begin{proof}
We use the notation as in Remark \ref{rem:red}.
Then we have 
\[ v(E)=v(E_{\mathrm{red}})+(M+N)v(A).\]
Recall that $v(A)$ is a spherical class. 
Therefore, the equality
\[ \langle v(A), v(E) \rangle=\langle v(A),  v(E_{\mathrm{red}}) \rangle  -2(M+N) \]
holds.

Applying $\Hom(A,-)$ to (\ref{eq:torsion M}), we have the exact sequence
\[
\begin{array}{r@{\quad\longrightarrow\quad}r@{\quad\longrightarrow\quad}l}
0 & \Hom(A,F) & \multicolumn{1}{l}{\Hom(A,E)\quad\overset{\varphi}{\longrightarrow}\quad\Hom(A, A^{\oplus M})} \\
  & \Ext^1(A,F) & \multicolumn{1}{l}{\Ext^1(A,E)\quad\longrightarrow\quad 0}
\end{array}
\]
If $\varphi$ is non-zero, then $A$ becomes a direct summand of $E$. By assumption, we have $\varphi=0$.
Therefore, the inequality
\begin{equation}\label{eq:M}
M \leq \ext^1(A,F) 
\end{equation}
holds.
Similarly, applying $\Hom(-,A)$ to (\ref{eq:torsion N}), we also have the exact sequence
\[ 0 \to \Hom(A^{\oplus N},A) \to \Ext^1(E_\mathrm{red},A) \to \Ext^1(F,A) \to 0 \]
by Remark \ref{rem:red}.
In particular, we have
\begin{equation}\label{eq:N}
\ext^1(A,E_\mathrm{red})=\ext^1(A,F)+N 
\end{equation}
by Serre duality. Combining (\ref{eq:M}) and (\ref{eq:N}), we obtain
\[ 0 \leq M+N \leq \ext^1(A,E_\mathrm{red})=\langle v(A), v(E_\mathrm{red}) \rangle. \]
Therefore, the inequality
\[ -\langle v(A), v(E_\mathrm{red}) \rangle \leq \langle v(A), v(E) \rangle \leq \langle v(A), v(E_\mathrm{red}) \rangle,  \]
which completes the proof.
\end{proof}

Next, we extend the support property for indecomposable objects in $\mcP_{\alpha_0}(0)$. 
Recall that we have $\mcP_{\alpha_0}(0)=\mcP_\alpha(0)$ for any $\alpha > \alpha_0$ by Remark \ref{rem:P(1) is abelian}. In particular, $\mcP_{\alpha_0}(0)$ is a finite length \textit{abelian} category (but not just quasi-abelian). Hence, we can take a Jordan--H{\"o}lder filtration for any non-zero object in $\mcP_{\alpha_0}(0)$. 
\begin{lem}\label{lem:lax support phase 0}
Let $E$ be a non-zero object in $\mcP_{\alpha_0}(0)$. Assume that $A$ is not a direct summand of $E$. Then we have
\[ \| v(E) \| < C|Z_{\alpha_0}(E)|. \]
\end{lem}
\begin{proof}
By Lemma \ref{lem:norm and red}, 
we have $\| v(E) \|\leq \| v(E_\mathrm{red})\|$. 
Note that we have the equalities
\[ \Hom(A,F)=0,~\Hom(E_\mathrm{red}/F,A)=0 \]
for any subobject $F$ of $E_\mathrm{red}$ in $\mcP_{\alpha_0}(0)$. 
In particular, $A$ is not a direct summand of $F$ or $E_\mathrm{red}/F$.

There is the maximum integer $k$ with a filtration 
\[ E^{(k)} \subset \cdots \subset E^{(1)} \subset E^{(0)}=E_\mathrm{red}\]
of $E_\mathrm{red}$ in $\mcP_{\alpha_0}(0)$ such that $E^{(i)}/E^{(i+1)}$ is a $\sigma_{\alpha_0}$-stable object not isomorphic to $A$ for $0 \leq i \leq k-1$. If $\Hom(E^{(k)},A)=0$ holds, the above filtration gives the Jordan-H\"older filtration of $E_\mathrm{red}$. Otherwise, we put \[ E^{(k+1)} \coloneqq E^{(k)}_\mathrm{red}.\] 
Since $\Hom(A, E^{(k)})=0$, we have $E^{(k+1)} \subset E^{(k)}$ and $E^{(k)}/E^{(k+1)} \simeq A^{\oplus N_k}$ for some $N_k \geq 0$.  Then there is an exact sequence
\[ 0 \to A^{\oplus N_k} \to E_\mathrm{red}/E^{(k+1)} \to E_\mathrm{red}/E^{(k)} \to 0 \]
 and \[ (E_\mathrm{red}/E^{(k+1)})_\mathrm{red}=E_\mathrm{red}/E^{(k)}\] holds.
The equality $\Hom(A,E_\mathrm{red}/E^{(k)})=0$ is deduced from the fact that no $\sigma_{\alpha_0}$-stable factor of $E_\mathrm{red}/E^{(k)}$ is isomorphic to $A$.
 Since $Z_{\alpha_0}(A)=0$, the equality 
 \[ 
 Z_{\alpha_0}(E_\mathrm{red}/E^{(k+1)})=Z_{\alpha_0}(E_\mathrm{red}/E^{(k)})
 \]
 holds. 
 By the triangle inequality and Lemma \ref{lem:norm and red}, we have
 \begin{align*}
 \| v(E_\mathrm{red}) \|&=\|v(E_\mathrm{red}/E^{(k+1)})+v(E^{(k+1)}) \| \\
 & \leq \|v(E_\mathrm{red}/E^{(k+1)})\| + \|v(E^{(k+1)})\| \\
 & \leq \|v(E_\mathrm{red}/E^{(k)}) \|+\|v(E^{(k+1)}) \|.
 \end{align*}
By Lemma \ref{lem:massless is semistable} and Remark \ref{rem:known support inequaility}, we obtain 
\[
\|v(E_\mathrm{red}/E^{(k)}) \| < C|Z_{\alpha_0}(E_\mathrm{red}/E^{(k)})|=CZ_{\alpha_0}(E_\mathrm{red}/E^{(k)}), \]
where $C>0$ is a constant as in Remark \ref{rem:known support inequaility}. 

We apply the same procedure to $E^{(k+1)}$ as we did to $E_\mathrm{red}$. If we continue in this way, the process will stop at some point. 
Then we obtain a  filtration 
\[ 0=E^{(n)} \subset E^{(n-1)} \subset \cdots \subset E^{(1)} \subset E^{(0)}=E_\mathrm{red}. \]
of $E_\mathrm{red}$ in $\mcP_{\alpha_0}(0)$ such that $E^{(i)}/E^{(i+1)}$ is a finite direct sum of $A$ or a $\sigma_{\alpha_0}$-stable object which is not isomorphic to $A$. 
Finally, we can obtain 
\[ \| v(E)\| < C|Z_{\alpha_0}(E)|  \]
by repeatedly applying the above arguments.
\end{proof}

\begin{lem}\label{lem:red and interval}
Assume that $E \in \mcP_{\alpha_0}(\phi-\epsilon, \phi+\epsilon)$ holds. 
Then we also have 
\[ E_{\mathrm{red}} \in \mcP_{\alpha_0}(\phi-\epsilon, \phi+\epsilon).\] 
\end{lem}
\begin{proof}
Recall that $E_\mathrm{red}$ is an object in $\mcB$.
There is a $\sigma_{\alpha_0}$-semistable object $E_1 \subset E_\mathrm{red}$ in $\mcB$ such that  the equality $\phi_{\alpha_0}(E_1)=\phi^+_{\alpha_0}(E_\mathrm{red})$ holds. 
If $\phi^+_{\alpha_0}(E_\mathrm{red})>\phi+\epsilon$ holds, then we have $\Hom(E_1, E_\mathrm{red})=0$. Since $E_1 \subset E_\mathrm{red}$, this is a contradiction. 
\end{proof}

\vspace{3mm}
\begin{proof}[Proof of Proposition \ref{prop:delta-lax support property}]

Take a massive indecomposable object $E$ in the quasi-abelian category $\mcP_{\alpha_0}(\phi-\epsilon,\phi+\epsilon)$. By Lemma \ref{lem:red and interval}, we also have 
\[ E_\mathrm{red} \in \mcP_{\alpha_0}(\phi-\epsilon,\phi+\epsilon). \]
Take the Harder--Narasimhan filtration of $E_\mathrm{red}$ with respect to $\sigma_{\alpha_0}$:
\[ 0=E_0 \subset E_1 \subset E_2 \subset \cdots \subset E_{n-1} \subset E_n = E_\mathrm{red} \]
with the $\sigma_{\alpha_0}$-semistable factor $A_i\coloneqq E_i/E_{i-1}$ for $1 \leq i \leq n$.
Take the maximum integer $k \geq 0$ satisfying $\phi_{\alpha_0}(A_k)>0$. Note that $E_k$ is contained in  $\mcP_{\alpha_0}(0,\phi+\epsilon)$.
On the other hand, the quotient $E_\mathrm{red}/E_k$ is contained in $\mcP_{\alpha_0}((\phi-\epsilon, 0])$. Since $\Hom(E_\mathrm{red}/E_k, A)=0$, the object $A$ is not a direct summand of $E_\mathrm{red}/E_k$. 
We now consider the object
\[ \hat{E} \coloneqq (E_\mathrm{red}/E_k)_\mathrm{red} \in \mcP_{\alpha_0}((\phi-\epsilon,0]).\]
Take the Harder--Narasimhan factors $\hat{A}_1, \cdots, \hat{A}_t$ of $\hat{E}$ with respect to $\sigma_{\alpha_0}$ such that $\phi_{\alpha_{0}}(\hat{A}_1)> \cdots >\phi_{\alpha_0}(\hat{A}_t)$.
If $\phi_{\alpha_0}(\hat{A}_1)<0$ holds, we have
\[ \|v(\hat{E})\|<\frac{C}{\cos(2\pi\epsilon )}|Z_{\alpha_0}(\hat{E}) | \]
by Lemma \ref{lem:lax support without massless}. Assume that $\phi_{\alpha_0}(\hat{A}_1)=0$ holds. Since $\Hom(A,\hat{E})=0$, the object $A$ is not a direct summand of $\hat{A}_1$. 
Hence we can apply Lemma \ref{lem:lax support phase 0} to obtain 
$\|v(\hat{A}_1)\| < C |Z_{\alpha_0}(\hat{A}_1)|$. 
Since we have $\phi-\epsilon < \phi(\hat{A}_i)<0$ for $i=2, \cdots, t$, 
the argument as in the proof of Lemma \ref{lem:lax support without massless} now proves that 
\[ \| v(\hat{E}) \|< \frac{C}{\cos(2\pi\epsilon)}|Z_{\alpha_0}(\hat{E}) |.\]
Therefore, we obtain
\begin{align*}
\| v(E) \| &\leq \|v(E_\mathrm{red}) \|\\
&\leq \| v(E_\mathrm{red}/E_k) \|+\|v(E_k) \|\\
&\leq \| v(\hat{E})\|+ \|v(E_k)\|\\
&<\frac{C}{\cos(2\pi\epsilon)}(|Z_{\alpha_0}(\hat{E})|+|Z_{\alpha_0}(E_k) |)\\
&=\frac{C}{\cos(2\pi\epsilon)}(|Z_{\alpha_0}(E_\mathrm{red}/E_k)|+|Z_{\alpha_0}(E_k) |)\\
&\leq \frac{C}{\cos^2(2\pi\epsilon)}|Z_{\alpha_0}(E_\mathrm{red}) |\\
&=\frac{C}{\cos^2(2\pi\epsilon)}|Z_{\alpha_0}(E)|.
\end{align*}
Thus, we have proved Proposition \ref{prop:delta-lax support property}.
\end{proof}

\bibliographystyle{alpha}
\bibliography{maths}

\end{document}